\documentclass[pdflatex,sn-mathphys-num]{sn-jnl}

\usepackage{graphicx}%
\usepackage{multirow}%
\usepackage{amsmath,amssymb,amsfonts}%
\usepackage{amsthm}%
\usepackage{mathrsfs}%
\usepackage[title]{appendix}%
\usepackage{xcolor}%
\usepackage{textcomp}%
\usepackage{manyfoot}%
\usepackage{booktabs}%
\usepackage{algorithm}%
\usepackage{algorithmicx}%
\usepackage{algpseudocode}%
\usepackage{listings}%

\usepackage[dvipsnames]{xcolor}
\usepackage{graphicx}
\usepackage{subcaption}
\graphicspath{ {figures/} }
\usepackage{amsmath, amssymb, thmtools}
\usepackage{enumitem}
\setlist[enumerate]{leftmargin=.5in}
\setlist[itemize]{leftmargin=.5in}
\usepackage{mathtools}
\usepackage{circuitikz}
\usepackage{tikzpagenodes}
\usepackage{setspace}
\usepackage{pgfplots}
\pgfplotsset{compat=1.18}
\usetikzlibrary{decorations.pathmorphing,patterns}

\definecolor{thesis_grey}{rgb}{0.17647, 0.19216, 0.21961} 
\colorlet{thesis_red}{BrickRed} 
\colorlet{thesis_green}{ForestGreen}
\definecolor{thesis_blue}{rgb}{0.062745, 0.56471, 0.83529}
\colorlet{thesis_pink}{WildStrawberry}
\colorlet{thesis_purple}{Mulberry}
\colorlet{thesis_yellow}{Goldenrod}

\theoremstyle{thmstyleone}%
\theoremstyle{thmstyletwo}%
\newtheorem{defn}{Definition}
\newtheorem{rem}{Remark}%
\newtheorem{cor}{Corollary}
\newtheorem{lem}{Lemma}
\newtheorem{prop}{Proposition}
\newtheorem{thm}{Theorem}

\theoremstyle{thmstylethree}%

\begin{document}

\title[Semi-Discrete Linear Hyperbolic Curvature Flow of Curves with Boundary]{Semi-Discrete Linear Hyperbolic Curvature flow}


\author*[1]{\fnm{James} \sur{McCoy}}\email{James.McCoy@rmit.edu.au}

\author[2]{\fnm{Rohan} \sur{St Hill}}\email{rohan.sthill@uon.edu.au}
\equalcont{These authors contributed equally to this work.}


\affil*[1]{\orgdiv{Department of Mathematical and Geospatial Sciences}, \orgname{School of Science, Royal Melbourne Institute of Technology}, \orgaddress{\street{124 La Trobe St}, \city{Melbourne}, \postcode{3000}, \state{Victoria}, \country{Australia}}}

\affil[2]{\orgdiv{School of Computer and Information Science}, \orgname{College of Engineering, Science and Environment, University of Newcastle}, \orgaddress{\street{University Drive}, \city{Callaghan}, \postcode{2308}, \state{NSW}, \country{Australia}}}



\abstract{We consider evolution of piecewise linear curves with boundary by linear hyperbolic curvature flows in which each interior vertex accelerates in proportion to its discrete curvature. The flow can be equally considered as a damped wave equation on a path graph or alternatively a damped mass-spring chain. The analysis rests on the spectral theory of the Dirichlet path Laplacian, a tridiagonal Toeplitz matrix and an associated quadratic eigenvalue problem. Given appropriate boundary behaviours we find corresponding self-similar solutions and more generally, given any specific boundary trajectories we write down general solutions using finite Fourier expansions.  We also consider the problem of evolving one piecewise linear curve to another using nonhomogeneous linear hyperbolic flow. Finally we discuss anisotropic variants and more general mass-spring systems.}




\maketitle

\section{Introduction}\label{sec:intro}
In \cite{BC07}, Chow and Glickenstein investigated the deformation of closed polygons by a first order linear system of ordinary differential equations (ODEs) which, with an appropriate definition of `normal' vector at each vertex, is a semi-discrete analogue of the curve shortening flow. This work has subsequently been generalised in several directions. Rademacher and Rademacher considered fully discrete flows in \cite{CR16}, while in \cite{DG17} Glickenstein and Liang studied related nonlinear semi-discrete flows and obtained convergence results in parallel with the smooth case. Recently in \cite{JM25}, the first author together with Meyer considered higher spatial order, polyharmonic analogues of Chow and Glickenstein's flow, together with a semi-discrete analogue of the Yau problem of evolving one curve to another by its curvature.

In the present article we are concerned with a \emph{hyperbolic} analogue of these flows, in which acceleration, rather than velocity, of each interior vertex is proportional to its discrete curvature, with an additional term proportional to the velocity that damps the motion. The continuous antecedents of such flows are well studied, with Gurtin and Podio-Guidugli \cite{MG91} introducing a damped hyperbolic law for the motion of plane interfaces, motivated by the oscillatory behavior of crystal melt interfaces that the parabolic law does not capture, and Rotstein, Brandon and Novick-Cohen \cite{HR99} studied its isotropic reduction. More general hyperbolic geometric flows were proposed by Kong and Liu \cite{DK07} and, in the case of motion by mean curvature, analysed by He, Kong and Liu \cite{CH09} and independently by LeFloch and Smoczyk \cite{PL08}, while  for plane curves the evolution under hyperbolic mean curvature flow and the formation of singularities were studied in \cite{KLW09,KW09}. The natural continuous counterpart of the flow considered here is the \emph{dissipative} hyperbolic mean curvature flow \cite{WLZ23,XW24,ZXW26}, which carries a velocity damping term, in the regime of small gradients it linearises to the damped wave equation, of which our system is the centered finite difference spatial discretisation. The closed polygon analogue of our hyperbolic flow was treated by the first author together with Meyer in \cite{JM25_1}.

The smooth hyperbolic flows are quasilinear hyperbolic problems and global explicit information about their solutions is quite scarce. This is one motivation for studying the linear semi-discrete counterparts. In the linear semi-discrete setting every solution can be written down explicitly. Thus, phenomena that are only partially understood for smooth flows, such as self-similar motion, the response to periodic forcing, etc. can be described explicitly. At the same time the flow described herein is exactly the wave equation on a path graph and the equation of motion of a damped mass-spring chain, so the results apply in each of these settings also.

The direct predecessor of the present work is \cite{JM24_1}, in which the present authors studied the semi-discrete \emph{parabolic} linear curvature flow of curves with boundary. That is, the flow of \cite{BC07} in the case where the two end points are prescribed rather than \emph{joined} together. There, self-similar scaling, translating and rotating solutions were constructed, an explicit representation formula for the solution with general initial data was obtained, and a semi-discrete Yau-type curvature difference flow with boundary was considered. The present article is a hyperbolic companion to \cite{JM24_1}. We provide analogues of each of the first order in time or `parabolic' results, here in the second order in time or `hyperbolic' setting. We establish several further results that have no counterpart in the parabolic case. The principal difference from the closed polygon flows of \cite{BC07, JM25, JM25_1} is that prescribing distinct end points replaces the circulant coefficient matrix of the periodic problem by a tridiagonal Toeplitz matrix with Dirichlet boundary conditions, whose eigenvalues and eigenvectors are given by trigonometric functions rather than complex exponentials, and the boundary data enters as a nonhomogeneous forcing term.

In passing from a first order to a second order system in time, the relevant spectral analysis becomes that of an associated quadratic eigenvalue problem, a structure that is classical in the study of damped vibrating systems; for more details we refer the reader to \cite{FT01} and the references contained therein. Indeed our system is precisely the equation of motion of a damped mass-spring chain with prescribed end displacements.

As in the parabolic case, our setting is closely related to the discretisations of evolving curves with boundary used to find approximate solutions of linear heat and wave-type equations by finite difference or finite element methods. Our piecewise linear curves may also be regarded as graphs $G=\{V,E\}$, with $V$ and $E$ the vertex and edge sets respectively, in which case the coefficient matrix is, up to sign, the graph Laplacian of a path and our equation is the wave equation for this graph (see for example \cite{FC96, DS13}). From this point of view the parabolic flow acts as a low-pass filter on the graph, whereas under the hyperbolic flow every underdamped mode oscillates and energy is dissipated through the damping term. Related ideas have recently been used to model the layers of a graph neural network as a diffusion \cite{BC21} or wave \cite{JY25} process on a graph.

The structure of the article is as follows. In Section~\ref{sec:setting} we fix notation, introduce the flow \eqref{eq:sdhcs} together with its first order reformulation, and record the energy dissipation law. The remainder of the section assembles the spectral information on which the later solution formulae rest. Section~\ref{sec:solutions} contains our results on the solutions. Section \ref{sec:selfsimilar} provides a complete classification of the self-similar solutions of plane curves, scaling, rotating, translation and the combinations thereof. Section~\ref{sec:timeperiodic} treats time periodic solutions, both damped and undamped. Section \ref{sec:generalsolutions} provides the representation formula for the solution with general initial data and arbitrary forcing. Section \ref{sec:curvaturedifference} constructs a hyperbolic curvature difference flow which for $\beta>0$, converges from any piecewise linear curve to any other with the same number of vertices. Section \ref{sec:variants} presents three variants of the system, being, coordinate-wise anisotropy, vertex-wise anisotropy and coefficient matrix variants, with the last being developed through two different damped mass-spring systems. Finally, Section \ref{sec:conclusion} summarises and discusses the results.
\section{Semi-discrete hyperbolic curvature flow of curves with boundary}\label{sec:setting}
We begin by describing our setting and fixing notation.
\begin{defn}[Piecewise linear curve]\label{defn:piecewise_linear_curve}
Let $p\geqslant 2$ and $n\geqslant 4$ be fixed integers. Then $\vec{X}$ denotes a \emph{piecewise linear curve} (see Figure \ref{fig:idea_example}) defined by joining consecutive points in the set $\vec{X}=(X_1,X_2,\dots,X_n)^\top$, where for $i=1,2,\dots,n$ each point $X_i\in\mathbb{R}^p$.
\end{defn}
\begin{figure}[hptb]
    \centering
    \includegraphics[width=0.25\linewidth]{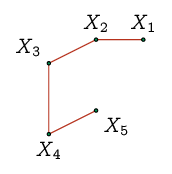}
    \caption{Example of piecewise linear curve in $\mathbb{R}^2$. The ordered collection of points are shown in green and the piecewise linear curve is shown in red.}
    \label{fig:idea_example}
\end{figure}
\begin{rem}
The set of points $\vec{X}$ can be represented as an $n\times p$ matrix. We denote by $X_i^j$ the $j$-th coordinate of vector $X_i$, where we use the regular Euclidean coordinates in $\mathbb{R}^p$.
\end{rem}

As in \cite{BC07,JM24_1} the inward pointing `normal', or discrete curvature, at an interior point is the second difference
\[N_i=(X_{i+1}-X_i)+(X_{i-1}-X_i)=X_{i-1}-2X_i+X_{i+1},\qquad i=2,3,\dots,n-1,\]
which approximates $\partial^2\gamma/\partial x^2$ at $X_i$ with unit mesh spacing. The parabolic flow \cite{ JM24_1} moves each interior vertex in the direction $N_i$. The semi-discrete hyperbolic flow instead \emph{accelerates} each interior vertex in this direction, subject to a damping proportional to its velocity. Writing $\beta\geqslant0$ for the damping parameter and prescribing the end points by $X_1=f_1(t)$ and $X_n=f_n(t)$, the interior points satisfy
\[\frac{\mathrm{d}^2X_i}{\mathrm{d}t^2}+\beta\frac{\mathrm{d}X_i}{\mathrm{d}t}=X_{i-1}-2X_i+X_{i+1},\qquad i=2,3,\dots,n-1.\]
Writing $\vec{U}=(X_2,X_3,\dots,X_{n-1})^\top$ for the vector of interior points, this system can be expressed in matrix form as follows.
\begin{defn}[Semi-discrete hyperbolic curvature flow]\label{defn:semi_discrete_hyperbolic}
    	The semi-discrete hyperbolic curvature flow is
	\begin{equation}\label{eq:sdhcs}
		\ddot{\vec{U}}+\beta\dot{\vec{U}}=A\vec{U}+\vec{f}(t),\tag{SDHF}
	\end{equation}
	where $\beta\geqslant0$, $\vec{f}(t)=e_1f_1(t)+e_{n-2}f_n(t)$, and $A$ is the $(n-2)\times(n-2)$ tridiagonal Toeplitz matrix
	\[A=
	\begin{bmatrix}
		-2 & 1 & & & 0\\
		1 & -2 & 1 & &\\
		& \ddots & \ddots & \ddots &\\
		& & 1 & -2 & 1\\
		0 & & & 1 & -2
	\end{bmatrix}
	=: \mathrm{tridiag}(1,-2,1).\]
	Here $e_1$ and $e_{n-2}$ denote the first and last standard basis vectors of $\mathbb{R}^{n-2}$, regarded as columns. Since we treat points as row vectors, $e_1f_1(t)$ is the $(n-2)\times p$ matrix whose first row is $f_1(t)$ and whose other rows vanish. Throughout we also assume that $f_1,f_n$ are locally integrable.
\end{defn}

The first order reformulation of \eqref{eq:sdhcs} doubles the dimension of the system but renders it a first order linear system, to which methods of matrix exponentials, variation of parameters and spectral decomposition apply directly.
\begin{rem}\label{rem:sdhcs_firstorder}
    	Setting $\vec{V}=\mathrm{d}\vec{U}/\mathrm{d}t$, the flow \eqref{eq:sdhcs} is equivalent to the first order system
	\begin{equation}\label{eq:sdhcs_firstorder}
		\frac{\mathrm{d}}{\mathrm{d}t}
		\begin{bmatrix}\vec{U}\\\vec{V}\end{bmatrix}
		=
		\begin{bmatrix}
        0 & I\\ 
        A & -\beta I
        \end{bmatrix}
		\begin{bmatrix}
        \vec{U}\\
        \vec{V}
        \end{bmatrix}
		+\vec{F}(t),\qquad
		\vec{F}(t)=\big(0,\;e_1f_1(t)+e_{n-2}f_n(t)\big)^\top.
	\end{equation}
	We can use this form when writing explicit solutions.
\end{rem}
\subsection{Approximation property}
The standard curve shortening flow is the $L^2$-gradient flow of the length functional $\mathcal{L}[\gamma]=\int|\gamma'|\;\mathrm{d}u$. While the parabolic flow follows the path of steepest descent, the hyperbolic curve shortening flow incorporates \emph{inertia}. For a curve $\gamma:[0,L]\times [0,T)\to\mathbb{R}^p$ the evolution equation is given by
\begin{equation*} 
	\frac{\partial^2 \gamma}{\partial t^2} + \beta \frac{\partial \gamma}{\partial t} = \kappa\mathbf{N} -\nabla p
\end{equation*}
where $\beta>0$ is the damping term, $\kappa$ is the scalar curvature, $\mathbf{N}$ is the unit normal, 
\[\nabla p=\left\langle \frac{\partial^2 \gamma}{\partial s\partial t},\frac{\partial \gamma}{\partial t}\right\rangle\mathbf{T}\]
is the cross derivative correction with $s$ for arc length and $\mathbf{T}$ being the unit tangent. Just as the linear heat equation approximates the parabolic curve shortening flow for graphs with small gradients, see for example \cite{JM24_1}, the linear damped wave equation approximates the hyperbolic curve shortening flow in the same regime, so
\begin{equation}\label{eq:dampedwave}
	\frac{\partial^2 \gamma}{\partial t^2} + \beta \frac{\partial \gamma}{\partial t} = \frac{\partial^2\gamma}{\partial x^2}.
\end{equation}
Associated with this flow is the total energy $\mathcal{E}$ which is defined as the sum of the \emph{kinetic} energy and \emph{potential} energy.  Assuming here the boundary points are fixed,
\begin{equation*} 
	\mathcal{E}[\gamma] = \frac{1}{2} \int_0^L \left| \frac{\partial \gamma}{\partial t} \right|^2\;\mathrm{d}x + \frac{1}{2}\int_0^L \left| \frac{\partial \gamma}{\partial x} \right|^2\;\mathrm{d}x.
\end{equation*}
Now taking the time derivative,
\begin{align*}
	\frac{\mathrm{d}}{\mathrm{d}t}\mathcal{E} &= \int_0^L \left\langle \frac{\partial \gamma}{\partial t},\frac{\partial^2\gamma}{\partial t^2}\right\rangle\;\mathrm{d}x+\int_0^L\left\langle \frac{\partial\gamma}{\partial x},\frac{\partial^2\gamma}{\partial t\partial x}\right\rangle\;\mathrm{d}x\\
	&=\int_0^L\left\langle \frac{\partial\gamma}{\partial t},\frac{\partial^2\gamma}{\partial t^2}\right\rangle\;\mathrm{d}x - \int_0^L\left\langle\frac{\partial^2\gamma}{\partial x^2},\frac{\partial\gamma}{\partial t}\right\rangle\;\mathrm{d}x\\
	&=\int_0^L\left\langle \frac{\partial\gamma}{\partial t},\frac{\partial^2\gamma}{\partial x^2}-\beta\frac{\partial \gamma}{\partial t}\right\rangle\;\mathrm{d}x-\int_0^L\left\langle \frac{\partial^2\gamma}{\partial x^2},\frac{\partial \gamma}{\partial t}\right\rangle\;\mathrm{d}x\\
	&=-\beta\int_0^L\left|\frac{\partial\gamma}{\partial t}\right|^2\;\mathrm{d}x
	=-2\beta \mathcal{K}\leqslant 0,
\end{align*}
where 
\begin{equation*}
	\mathcal{K} = \frac{1}{2}\int_0^L\left|\frac{\partial \gamma}{\partial t}\right|^2\;\mathrm{d}x
\end{equation*}
is the \emph{kinetic} energy of the system. 

Now discretising the interval $[0,L]$ at the $n$ equally spaced points $x_i=(i-1)\Delta x$ where $\Delta x=L/(n-1)$ using a centred finite difference for the Laplacian in \eqref{eq:dampedwave} we obtain the system of equations
\begin{equation*}
	\frac{\mathrm{d}^2 u_i}{\mathrm{d}t^2} + \beta\frac{\mathrm{d}u_i}{\mathrm{d}t} = \frac{u_{i-1} - 2u_i + u_{i+1}}{(\Delta x)^2}.
\end{equation*}
With unit mesh spacing $\Delta x=1$, that is $L=n-1$, this is exactly \eqref{eq:sdhcs} for the interior points. For a general mesh width, the change of time variable $t=\Delta x\tilde{t}$ transforms it into \eqref{eq:sdhcs} in the variable $\tilde{t}$ with damping coefficient $\beta\Delta x$ in place of $\beta$, thus the unit spacing entails no loss of generality.

\begin{prop}[Approximation Property]\label{prop:discreteenergy}
	Let $\beta\geqslant0$ and let $U(t)$ be a solution to the semi-discrete hyperbolic flow \eqref{eq:sdhcs} with time independent boundary conditions (that is $f_1(t)\equiv X_1(0)$ and $f_n(t)\equiv X_n(0)$). Define the total discrete energy $E(t)$ as the sum of the discrete kinetic energy $K$ and the discrete Dirichlet energy $W$,
	\begin{equation}\label{eq:dirichletenergy}
		E = K + W = \frac{1}{2} \sum_{i=2}^{n-1}\left| \frac{dX_i}{dt} \right|^2 + \frac{1}{2} \sum_{i=2}^{n} |X_i - X_{i-1}|^2
	\end{equation}
	Thus the total energy is non-increasing and satisfies the dissipation law
	\begin{equation} 
		\frac{\mathrm{d}}{\mathrm{d}t} E(t) = -2\beta K \leqslant 0.
	\end{equation}
\end{prop}
\begin{proof}
We begin by differentiating the total energy $E$ with respect to time. First considering the rate of change of the discrete kinetic energy $K$,
	\begin{equation*}
		\frac{\mathrm{d}}{\mathrm{d}t}K = \frac{\mathrm{d}}{\mathrm{d}t}\left(\frac{1}{2}\sum_{i=2}^{n-1}\left|\frac{\mathrm{d}X_i}{\mathrm{d}t}\right|^2\right)
		= \sum_{i=2}^{n-1}\left\langle\frac{\mathrm{d}X_i}{\mathrm{d}t},\frac{\mathrm{d}^2X_i}{\mathrm{d}t^2}\right\rangle.
	\end{equation*}
	Next we differentiate the discrete Dirichlet energy $W$ and sum by parts, the boundary terms vanish since the boundary points are fixed such that $\mathrm{d}X_1/\mathrm{d}t=\mathrm{d}X_n/\mathrm{d}t=0$, 
	\begin{multline*}
		\frac{\mathrm{d}}{\mathrm{d}t}W = \frac{\mathrm{d}}{\mathrm{d}t}\left(\frac{1}{2}\sum_{i=2}^{n}\left|X_i-X_{i-1}\right|^2\right)
		= \sum_{i=2}^{n}\left\langle X_i-X_{i-1},\frac{\mathrm{d}X_i}{\mathrm{d}t}-\frac{\mathrm{d}X_{i-1}}{\mathrm{d}t}\right\rangle\\
		=\sum_{i=2}^{n-1}\left\langle\left(X_i-X_{i-1}\right)-\left(X_{i+1}-X_i\right),\frac{\mathrm{d}X_i}{\mathrm{d}t}\right\rangle
		=-\sum_{i=2}^{n-1}\left\langle X_{i-1}-2X_i+X_{i+1},\frac{\mathrm{d}X_i}{\mathrm{d}t}\right\rangle.
	\end{multline*}
	Substituting these into the total time derivative,
	\begin{equation}\label{eq:discretetotal}
		\frac{\mathrm{d}}{\mathrm{d}t}E = \sum_{i=2}^{n-1}\left\langle\frac{\mathrm{d}X_i}{\mathrm{d}t},\frac{\mathrm{d}^2X_i}{\mathrm{d}t^2} - \left(X_{i-1}-2X_i+X_{i+1}\right)\right\rangle.
	\end{equation}
	From \eqref{eq:sdhcs} we have that
	\[\frac{\mathrm{d}^2X_i}{\mathrm{d}t^2} = \left(X_{i-1}-2X_i+X_{i+1}\right) - \beta \frac{\mathrm{d}X_i}{\mathrm{d}t}.\]
	Substituting this into \eqref{eq:discretetotal} we obtain
	\begin{equation*}
		\frac{\mathrm{d}}{\mathrm{d}t}E = \sum_{i=2}^{n-1}\left\langle \frac{\mathrm{d}X_i}{\mathrm{d}t},-\beta\frac{\mathrm{d}X_i}{\mathrm{d}t}\right\rangle
		=-\beta\sum_{i=2}^{n-1}\left|\frac{\mathrm{d}X_i}{\mathrm{d}t}\right|^2.
	\end{equation*}
	Since
	\[K = \frac{1}{2}\sum_{i=2}^{n-1}\left|\frac{\mathrm{d}X_i}{\mathrm{d}t}\right|^2,\]
	we conclude
	\[\frac{\mathrm{d}}{\mathrm{d}t} E=-2\beta K\leqslant 0.\]
\end{proof}

\begin{rem}[Time-Dependent Boundary Conditions]
	The strict energy dissipation property derived in Proposition \ref{prop:discreteenergy} relies on the boundary points being fixed, i.e. $\mathrm{d}X_1/\mathrm{d}t = \mathrm{d}X_n/\mathrm{d}t = 0$. If the boundary conditions are time-dependent, that is, $X_1(t) = f_1(t)$ and $X_n(t) = f_n(t)$ with $f_1,f_n\in C^1$, the system is no longer isolated, and the boundary points perform  \emph{work} on the curve.

In this case, the summation by parts in the proof produces the non-vanishing boundary terms 
\[\left\langle X_n-X_{n-1},\frac{\mathrm{d}X_n}{\mathrm{d}t}\right\rangle - \left\langle X_2-X_1,\frac{\mathrm{d}X_1}{\mathrm{d}t}\right\rangle\]
with $X_1=f_1$ and $X_n=f_n$. The energy evolution equation becomes:
\begin{equation*}
	\frac{\mathrm{d}E}{\mathrm{d}t} = -2\beta K + \left\langle f_n(t) - X_{n-1}, \frac{\mathrm{d}}{\mathrm{d}t}f_n(t) \right\rangle - \left\langle X_2 - f_1(t), \frac{\mathrm{d}}{\mathrm{d}t}f_1(t) \right\rangle.
\end{equation*}
The additional terms represent the \textit{external power} supplied to the system. Consequently, the total energy is not necessarily non-increasing, it may increase if the boundary movement injects energy into the system faster than the damping term $-2\beta K$ can dissipate it.
\end{rem}

\begin{figure}[htb]
	\centering
	\begin{subfigure}[t]{0.459\textwidth}
		\includegraphics[width=\textwidth]{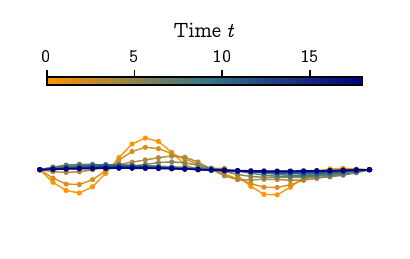}
		\caption{Evolution of the semi-discrete hyperbolic linear curvature flow with $\beta=0.6$. Individual time steps are shown superimposed.}
	\end{subfigure}
	\hfill
	\begin{subfigure}[t]{0.49\textwidth}
		\includegraphics[width=\textwidth]{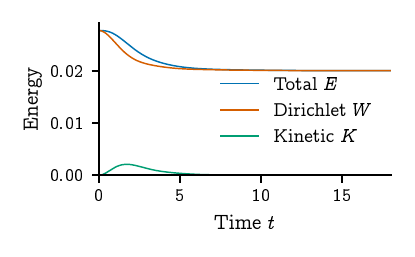}
		\caption{Evolution of the energy of the curve shown in sub-figure (a).}
	\end{subfigure}
	\begin{subfigure}[t]{0.459\textwidth}
		\includegraphics[width=\textwidth]{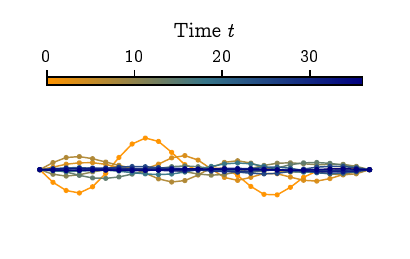}
		\caption{Evolution of the semi-discrete hyperbolic linear curvature flow with $\beta=0.15$. Individual time steps are shown superimposed.}
	\end{subfigure}
	\hfill
	\begin{subfigure}[t]{0.49\textwidth}
		\includegraphics[width=\textwidth]{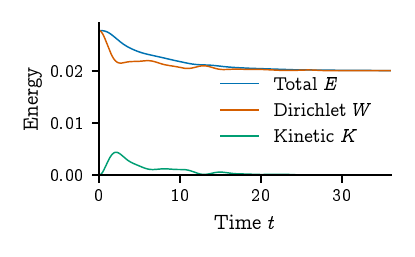}
		\caption{Evolution of the energy of the curve shown in sub-figure (c).}
	\end{subfigure}

	\caption[Energy decay of curve undergoing evolution]{Evolution of a curve under \eqref{eq:sdhcs} with fixed end points, showing the decay of energy. In (a) and (c) the timesteps are coloured from yellow through teal to navy blue as $t$ increases. (b) and (d) show the total energy $E$, the Dirichlet energy $W$ and the kinetic energy $K$. Parameters of the curves are given in Appendix \ref{sec:app}.}
	\label{fig:energy}
\end{figure}

Figure \ref{fig:energy} illustrates this dissipation law on one initial curve with constant boundary and two values of the damping coefficient $\beta$. Using Remark \ref{rem:damping} we know the first few critically damped modes are $\beta^*_1\approx0.251,\beta^*_2\approx0.501$ and $\beta^*_3\approx0.750$.  With $\beta=0.6$ (panels (a)-(b)) the two slowest modes are overdamped while the higher modes remain underdamped. Whereas, with $\beta=0.15<\beta^*_1$ (panels (c)-(d)) all modes are underdamped and the lighter damping permits several overshoots before settling. Notice however in both cases the total energy is non-increasing

\subsection{Matrix properties}\label{sec:mat}
The explicit solution formulae for \eqref{eq:sdhcs} developed in later sections rest on a complete spectral analysis of two matrices; the spatial operator $A=\mathrm{tridiag}(1,-2,1)$ governing the curvature coupling and the block matrix
\[
M=
\begin{bmatrix}
	0 & I\\
	A & \mathrm{diag}(-\beta)\\
\end{bmatrix}
\]
arising from the reduction of order for \eqref{eq:sdhcs}. The eigenvalues and eigenvectors of $A$ are classical and are recorded in Lemma \ref{lem:etri} for convenience. The block matrix $M$ is more nuanced, its eigenvalues solve a quadratic eigenvalue problem of the form
\[\lambda^2+\beta\lambda-\mu_k=0\]
in terms of the eigenvalues $\mu_k$ of $A$, so the $n-2$ eigenvalues of $A$ give rise to the $2(n-2)$ eigenvalues of $M$, counted with algebraic multiplicity. We compute these in two stages of increasing generality.

Proposition \ref{prop:esystem} treats the case relevant to \eqref{eq:sdhcs}, where the damping matrix is the scalar multiple $\mathrm{diag}(-\beta)$, and gives the explicit eigenvalues and eigenvectors of $M$ purely in terms of the eigenpairs of $A$. We then record a more general result in Proposition \ref{prop:blockdiaggeneral}, which handles arbitrary simultaneously diagonalisable stiffness and damping matrices, via a different proof technique. While this is not strictly necessary for the analysis of \eqref{eq:sdhcs}, Proposition \ref{prop:blockdiaggeneral} subsumes Proposition \ref{prop:esystem} as a special case and is useful in its own right. We present both because they exemplify complementary methods to approach the underlying quadratic eigenvalue problem and the more general case may be useful in certain dynamical system problems. Finally, Proposition \ref{prop:matexp} solves \eqref{eq:sdhcs} mode by mode in the eigenbasis of $A$ and expresses $e^{Mt}$ through two functions $\mathcal{C}$ and $\mathcal{S}$ of $A$, valid in every damping regime, including the critically damped case in which $M$ cannot be diagonalised.

We begin by collecting the relevant tridiagonal Toeplitz results in a single place.
\begin{lem}[Tridiagonal Toeplitz matrix eigenvalues and eigenvectors]\label{lem:etri}
	Let $\mathrm{tridiag}(c,a,b)$ denote the $m\times m$ tridiagonal Toeplitz matrix with sub-diagonal entries $c$, diagonal entries $a$ and super-diagonal entries $b$, where $a,b,c$ are real and $bc>0$. Then the eigenvalues of $\mathrm{tridiag}(c,a,b)$ are the $m$ distinct real numbers
	\[\mu_k = a + 2b\left(\frac{c}{b}\right)^\frac{1}{2}\cos\left(\frac{k\pi}{m+1}\right)\text{ for }k=1,2,\dots,m\]
	and the corresponding eigenvectors are
	\begin{multline*}
    u_k = \left[ \left(\frac{c}{b}\right)^\frac{1}{2}\sin\left(\frac{k\pi}{m+1}\right),\frac{c}{b}\sin\left(\frac{2k\pi}{m+1}\right),\left(\frac{c}{b}\right)^\frac{3}{2}\sin\left(\frac{3k\pi}{m+1}\right), \right. \\
    \left. \dots,
    \left(\frac{c}{b}\right)^\frac{m}{2}\sin\left(\frac{mk\pi}{m+1}\right)\right]^\top,
\end{multline*}
where $(c/b)^{1/2}>0$.
\end{lem}
\begin{proof}
	See for example \cite[Chapter 3]{SM79}.
\end{proof}
Many of the explicit computations in later sections require not the spectrum of $A$ but its inverse. Thus a closed form expression for the inverse is therefore valuable. The following lemma records the standard formula in the broader setting of an arbitrary symmetric tridiagonal Toeplitz matrix. 
\begin{lem}[Symmetric tridiagonal matrix inverse]\label{lem:inverse}
Let $a$ be real, then the $m\times m$ matrix $\mathrm{tridiag}(1,a,1)$ is invertible if and only if
\[a\notin \{2\cos\left(\frac{k\pi}{m+1}\right):k=1,2,\dots,m\}\subset (-2,2),\]
and then for $i,j=1,2,\dots,m$, its inverse is
	\[\left(\mathrm{tridiag}(1,a,1)^{-1}\right)_{ij} = 
	\begin{cases}
		-\frac{\cosh((m+1-|j-i|)\lambda)-\cosh((m+1-i-j)\lambda)}{2\sinh\lambda\sinh((m+1)\lambda)}	& a < -2\\
		(-1)^{i+j}\frac{\cosh((m+1-|j-i|)\lambda)-\cosh((m+1-i-j)\lambda)}{2\sinh\lambda\sinh((m+1)\lambda)}& a > 2\\
		-(-1)^{i+j}\frac{\cos((m+1-|j-i|)\lambda)-\cos((m+1-i-j)\lambda)}{2\sin\lambda\sin((m+1)\lambda)}& -2 < a < 2\\
        -\frac{\min(i,j)((m+1)-\max(i,j))}{m+1} & a=-2\\
        (-1)^{i+j}\frac{\min(i,j)((m+1)-\max(i,j))}{m+1} & a=2
	\end{cases}\]
	where
	\[
	\begin{cases}
		a=-2\cosh\lambda & \lambda>0,\quad a< -2\\
		a=2\cosh\lambda & \lambda>0,\quad a> 2\\
		a=2\cos\lambda & 0<\lambda<\pi,\quad-2<a<2.
	\end{cases}\]
\end{lem}
\begin{proof}
    See \cite{GH96} for example.
\end{proof}

Though Lemma \ref{lem:inverse} is sufficient for our purposes, it is worth contextualising this within the spectral graph theory setting, where the inverse can be recognised as the Green's function for the path graph Laplacian with Dirichlet boundary conditions. The case $a=-2$ of Lemma \ref{lem:inverse} gives $A^{-1}$ where $m=n-2$ is the number of interior vertices.
\begin{rem}
	The case $a=-2$ of Lemma \ref{lem:inverse} is the path graph specialisation of the discrete Green's function studied by Chung and Yau \cite{FC00}. In their framework, the discrete Green's function $\mathcal{G}$ of an induced subgraph $S$ of a graph $\Gamma$ is the inverse of the Dirichlet Laplacian $\mathcal{L}_S$ restricted to $S$, and plays the same role for diffusion type problems on graphs as the classical Green's function plays for the corresponding continuous partial differential equation. For the path $P_m$ with vertex set $\{1,2,\dots,m\}$ and boundary $\{0,m+1\}$, Chung and Yau \cite[Theorem 3]{FC00} derive
	\[\mathcal{G}(i,j) = \frac{2}{m+1}\min(i,j)((m+1)-\max(i,j))\]
	for the normalised Laplacian 
	\[\mathcal{L}_S=\frac{1}{2}L_S\]
	where $L_S=\mathrm{tridiag}(-1,2,-1)$ is the Dirichlet Laplacian on the interior vertices and the factor $1/2$ arises because each interior vertex of the path has degree 2. To convert this to our setting, observe that $A=-L_S$ so 
	\[L^{-1}_S=\frac{1}{2}\mathcal{G} \implies A^{-1}=-L_S^{-1}=-\frac{1}{2}\mathcal{G}.\]
\end{rem}

We turn now to the spectral analysis of the block matrix $M$.
\begin{prop}[Block matrix eigendecomposition]\label{prop:esystem}
	Let $a,b,c$ be real with $b\, c>0$ and let $\beta\geqslant 0$. Then, counted with algebraic multiplicity, the  eigenvalues of the $2m\times2m$ matrix
	\[\begin{bmatrix}
		0 & I\\
		\mathrm{tridiag}(c,a,b) & \mathrm{diag}(-\beta)\\
	\end{bmatrix}
	\]
	are
	\[\lambda_k^\pm = \frac{-\beta\pm\sqrt{\beta^2+4\mu_k}}{2}\text{ for }k=1,2,\dots,
m,\]
	where
	\[\mu_k = a + 2b\left(\frac{c}{b}\right)^\frac{1}{2}\cos\left(\frac{k\pi}{m+1}\right)\]
	and the corresponding eigenvectors are
	\[v_k^\pm = (u_k ,\lambda_k^\pm u_k)^\top \text{ for }k=1,2,\dots,m,\]
	where
	\begin{multline*}
    u_k = \left[ \left(\frac{c}{b}\right)^\frac{1}{2}\sin\left(\frac{k\pi}{m+1}\right),\frac{c}{b}\sin\left(\frac{2k\pi}{m+1}\right),\left(\frac{c}{b}\right)^\frac{3}{2}\sin\left(\frac{3k\pi}{m+1}\right), \right. \\
    \left. \dots,
    \left(\frac{c}{b}\right)^\frac{m}{2}\sin\left(\frac{mk\pi}{m+1}\right)\right]^\top.
\end{multline*}

These $2m$ eigenvalues are distinct, except if $\beta^2+4\mu_k=0$ for some $k$, then $\lambda_k^+=\lambda_k^-=-\beta/2$ has algebraic multiplicity two and geometric multiplicity one. Consequently the matrix is diagonalisable if and only if $\beta^2+4\mu_k\neq0$ for every $k$.
\end{prop}
\begin{proof}
Let
\[B = \begin{bmatrix}
		0 & I\\
		\mathrm{tridiag}(c,a,b) & \mathrm{diag}(-\beta)\\
	\end{bmatrix},
	\]
	then
	\begin{align*}
		\det(B-I\lambda) &= \det\left(\mathrm{diag}(-\lambda)\mathrm{diag}(-\beta-\lambda)-\mathrm{tridiag}(c,a,b)\right)\\
		&= \det\left(\mathrm{diag}(\lambda^2 + \beta\lambda)-\mathrm{tridiag}(c,a,b)\right).
	\end{align*}
	Now let $\mu=\lambda^2+\beta\lambda$ such that we have
	\[0=\mathrm{det}\left(\mathrm{tridiag}(c,a,b)-\mu I\right).\]
	We note from Lemma \ref{lem:etri} that the solutions are
	\[\mu_k = a + 2b\left(\frac{c}{b}\right)^\frac{1}{2}\cos\left(\frac{k\pi}{m+1}\right)\text{ for }k=1,2,\dots,m.\]
	So
	\begin{equation*}
		\lambda^2 + \beta\lambda = \mu_k
		\implies \lambda_k^\pm = \frac{-\beta\pm\sqrt{\beta^2+4\mu_k}}{2}\text{ for }k=1,2,\dots,m.
	\end{equation*}
	Now for the eigenvectors
	\begin{align*}
		Bv=\lambda v
		\iff 
		\begin{bmatrix}
			0 & I \\
			\mathrm{tridiag}(c,a,b) & \mathrm{diag}(-\beta) \\
		\end{bmatrix}
		\begin{bmatrix}
			u\\
			w\\
		\end{bmatrix}
		=\lambda
		\begin{bmatrix}
			u\\
			w\\
		\end{bmatrix}.
	\end{align*}
	We can now write this as
	\begin{align*}
		&\mathrm{tridiag}(c,a,b)u + \lambda \mathrm{diag}(-\beta) u=\lambda^2u
		\iff 0 = \lambda^2u+\beta\lambda u - \mathrm{tridiag}(c,a,b)u\\
		&\implies \mathrm{tridiag}(c,a,b)u = (\lambda^2+\beta\lambda)u
		=\mu u.
	\end{align*}
	We note from Lemma \ref{lem:etri} that the solutions are
	\begin{multline*}
    u_k = \left[ \left(\frac{c}{b}\right)^\frac{1}{2}\sin\left(\frac{k\pi}{m+1}\right),\frac{c}{b}\sin\left(\frac{2k\pi}{m+1}\right),\left(\frac{c}{b}\right)^\frac{3}{2}\sin\left(\frac{3k\pi}{m+1}\right), \right. \\
    \left. \dots,
    \left(\frac{c}{b}\right)^\frac{m}{2}\sin\left(\frac{mk\pi}{m+1}\right)\right]^\top.
\end{multline*}
Thus the eigenvectors of $B$ are
	\[v_k^\pm = (u_k ,\lambda^\pm_ku_k)^\top \text{ for }k=1,2,\dots,m.\]
Since the eigenvalues $\mu_1,\mu_2,\dots,\mu_m$ of $\mathrm{tridiag}(c,a,b)$ are distinct, each eigenvalue $\lambda$ of $B$ belongs to exactly one $k$, namely the one with $\lambda^2+\beta\lambda=\mu_k$. So each eigenvalue $\lambda$ of $B$ has, up to scale, the single eigenvector $(u_k,\lambda u_k)^\top$, and two of the $2m$ eigenvalues $\lambda_k^\pm$ can coincide only as $\lambda_k^+=\lambda^-_k=-\beta/2$ when $\beta^2+4\mu_k=0.$ Hence $B$ is diagonalisable if and only if $\beta^2+4\mu_k\neq0$ for every $k$.

\end{proof}

\begin{rem}\label{rem:damping}
    For $\mathrm{tridiag}(c,a,b)=A$, Lemma \ref{lem:etri} gives 
        \[\mu_k=-2+2\cos\left(\frac{k\pi}{m+1}\right)=-4\sin^2\left(\frac{k\pi}{2(m+1)}\right).\]
        We can then define the critical values
        \[\beta_k^*:=2\sqrt{-\mu_k}=4\sin\left(\frac{k\pi}{2(m+1)}\right),\text{ for }k=1,2,\dots,m,\]
        so that $0<\beta_1^*<\beta_2^*<\dots<\beta_m^*<4.$ The $k$-th mode is
        \begin{enumerate}[label=(\roman*)]
            \item \emph{Underdamped} if $\beta<\beta_k^*$ with $\lambda^\pm_k=-\beta/2\pm i\sqrt{-\mu_k-\beta^2/4}$.
            \item \emph{Critically damped} if $\beta=\beta_k^*$ with $\lambda^\pm_k=-\beta/2$.
            \item \emph{Overdamped} if $\beta>\beta_k^*$ with two distinct negative real roots.
        \end{enumerate}
        Since the $\beta_k^*$ are distinct, at most one mode is critically damped for a given $\beta$.
\end{rem}
Proposition \ref{prop:esystem} furnishes spectral information required to solve \eqref{eq:sdhcs} in its presented form where the damping enters as a scalar $\beta$. However, we now record a more general result which subsumes Proposition \ref{prop:esystem} that proceeds by a different method. The proof in this generality cannot use the block reduction of Proposition \ref{prop:esystem}. Instead, working directly with the eigenvalue equation $Bv=\lambda v$ and substituting a shared eigenvector reduces the problem to a scalar quadratic in each mode separately.

\begin{prop}[Simultaneously diagonalisable eigendecomposition]\label{prop:blockdiaggeneral}
	Let $T$ and $D$ be $m\times m$ matrices that are simultaneously diagonalisable, that is, they share a common eigenbasis $\{u_k\}_{k=1}^m$ with $Tu_k=\mu_k u_k$ and $Du_k=\omega_ku_k$. Then, counted with algebraic multiplicity, the $2m\times 2m$ matrix
	\[B = 
	\begin{bmatrix}
		0 & I \\
		T & -D
	\end{bmatrix}\]
	has eigenvalues
	\[\lambda_k^\pm = \frac{-\omega_k\pm\sqrt{\omega_k^2+4\mu_k}}{2},\quad k=1,2,\dots,m,\]
	and the corresponding eigenvectors are
	\[v_k^\pm = (u_k,\lambda^\pm_ku_k)^\top,\quad k=1,2,\dots,m.\]
\end{prop}
\begin{proof}
	Let $v=(u,w)^\top$ be an eigenvector of $B$ with eigenvalue $\lambda$. Then the eigenvalue equation $Bv=\lambda v$ gives the block system
	\begin{align}
		w&=\lambda u,\label{eq:blockdiagprop1}\\
		Tu-Dw&=\lambda w. \label{eq:blockdiagprop2}
	\end{align}
	Substituting \eqref{eq:blockdiagprop1} into \eqref{eq:blockdiagprop2} we obtain
	\[Tu-\lambda Du=\lambda^2u\implies Tu=\lambda^2u+\lambda Du.\]
	Since $T$ and $D$ share the eigenbasis $\{u_k\}$ we set $u=u_k$ and use $Tu_k=\mu_ku_k, Du_k=\omega_ku_k$ to obtain
	\[\mu_ku_k=\lambda^2u_k+\lambda\omega_ku_k.\]
	Since $u_k\neq0$ we obtain the scalar quadratic
	\[\lambda^2+\omega_k\lambda-\mu_k=0.\]
	Hence
	\[\lambda^\pm_k=\frac{-\omega_k\pm\sqrt{\omega^2_k+4\mu_k}}{2}.\]
	The corresponding eigenvectors are $v_k^\pm=(u_k,\lambda^\pm_ku_k)^\top$ from \eqref{eq:blockdiagprop1}. 
\end{proof}
The hypothesis of simultaneous diagonalisability is satisfied automatically in an important special case where both the stiffness and damping matrices are tridiagonal Toeplitz with the same ratio of off-diagonal entries. This is common in many dynamical systems problems. 
\begin{cor}[Simultaneously diagonalisable eigendecomposition]\label{cor:blockdiagspecific}
	When $T$ and $D$ from Proposition \ref{prop:blockdiaggeneral} are both real tridiagonal Toeplitz matrices, that is, $T=\mathrm{tridiag}(c,a,b)$ and $D=\mathrm{tridiag}(c',a',b')$, with  $bc>0,b'c'>0$ and $c'/b'=c/b$ (so that they share the same eigenvectors), the eigenvalues are
	\[\lambda_k^\pm = \frac{-\omega_k\pm\sqrt{\omega^2_k+4\mu_k}}{2}\]
	where
	\begin{equation*}
		\mu_k = a+2b\left(\frac{c}{b}\right)^\frac{1}{2}\cos\left(\frac{k\pi}{m+1}\right), \qquad
		\omega_k = a' + 2b'\left(\frac{c'}{b'}\right)^\frac{1}{2}\cos\left(\frac{k\pi}{m+1}\right),
	\end{equation*}
	and the eigenvectors are
	\[v^\pm_k=(u_k,\lambda_k^\pm u_k)\text{ where }(u_k)_j=\left(\frac{c}{b}\right)^\frac{j}{2}\sin\left(\frac{jk\pi}{m+1}\right).\]
\end{cor}
\begin{proof}
	The eigenvectors of the tridiagonal Toeplitz matrix $\mathrm{tridiag}(c,a,b)$ are by Lemma \ref{lem:etri}
	\[(u_k)_j = \left(\frac{c}{b}\right)^{\frac{j}{2}} \sin\left(\frac{jk\pi}{m+1}\right)\text{ for }j=1,2,\dots,m.\]
	These eigenvectors depend only on the ratio $c/b$ and the dimension $m$. Since $c'/b'=c/b$ the matrices $T$ and $D$ share the same eigenvectors $\{u_k\}_{k=1}^m$. By Lemma \ref{lem:etri} the corresponding eigenvalues are
	\[Tu_k=\mu_ku_k \text{ and }Du_k=\omega_ku_k\]
	where
	\[\mu_k=a+2b\left(\frac{c}{b}\right)^\frac{1}{2}\cos\left(\frac{k\pi}{m+1}\right)\text{ and }\omega_k=a'+2b'\left(\frac{c'}{b'}\right)^\frac{1}{2}\cos\left(\frac{k\pi}{m+1}\right).\]
	Since $T$ and $D$ are simultaneously diagonalisable with common eigenbasis $\{u_k\}$ we can utilise Proposition \ref{prop:blockdiaggeneral}. Applying it directly yields the eigenvalues 
	\[\lambda_k^\pm=\frac{-\omega_k\pm\sqrt{\omega_k^2+4\mu_k}}{2}\]
	and eigenvectors $v_k^\pm=(u_k,\lambda_k^\pm u_k)^\top$ for $k=1,2,\dots,m$.
\end{proof}

As in Proposition  \ref{prop:esystem}, if $\omega_k^2+4\mu_k=0$ for some $k$, then $v^+_k=v_k^-$ in Proposition \ref{prop:blockdiaggeneral} and Corollary \ref{cor:blockdiagspecific} and the eigenvectors do not form a basis. This critically damped case is treated in the same way as for \eqref{eq:sdhcs} below.

We conclude this section with a representation of $e^{Mt}$ that is valid in every damping regime. Although $M$ is not diagonalisable when a mode is critically damped, the orthonormal eigenvectors of $A$ always decouple \eqref{eq:sdhcs} into $m$ scalar damped oscillators and these can always be solved by a single formula.

For $\mu<0$ and $\beta\geqslant0$, let $G(\cdot;\mu)$ be the solution of $g''+\beta g'-\mu g=0$ with $g(0)=0,g'(0)=1$. Writing $g=e^{-\beta t/2}w$ removes the damping, since,
\[g''+\beta g'-\mu g=e^{-\beta t/2}(w''-\omega w)\text{ where }\omega:=\mu+\frac{\beta^2}{4},\]
with $w(0)=0$ and $w'(0)=1$. Differentiating $w''=\omega w$ repeatedly gives $w^{(j+2)}=\omega w^{(j)}$, so $w^{(2j)}(0)=0$ and $w^{(2j+1)}(0)=\omega^j$, and the Taylor series of $w$ centred at $t=0$ gives
\begin{equation}\label{eq:taylorsingle}
    G(t;\mu)=e^{-\beta t/2}\sum_{j=0}^\infty \omega^j\frac{t^{2j+1}}{(2j+1)!}=
    \begin{cases}
        \frac{e^{-\beta t/2}\sinh\left(\sqrt{\omega}t\right)}{\sqrt{\omega}}&\omega>0\\
        te^{-\beta t/2}&\omega=0\\
        \frac{e^{-\beta t/2}\sin\left(\sqrt{-\omega}t\right)}{\sqrt{-\omega}}&\omega<0.
    \end{cases}
\end{equation}
For $|\omega|\leqslant R$ and $|t|\leqslant T$ the series is dominated by
\[\sum_{j=0}^\infty R^j\frac{T^{2j+1}}{(2j+1)!}<\infty,\]
so it converges uniformly on bounded sets and $G$ is continuous in $t,\mu$ and $\beta$. By Remark \ref{rem:damping}, the three cases $\omega<0,\omega=0$ and $\omega>0$ are the underdamped, critically damped and overdamped regimes. Note also that the solution with $g(0)=1$ and $g'(0)=0$ is $G'+\beta G$ since $(G'+\beta G)'(0)=G''(0)+\beta G'(0)=\mu G(0)=0$.

To assemble the modes let $Q$ be the $m\times m$ matrix whose columns are the orthonormal eigenvectors of $A$ given by Lemma \ref{lem:etri} then $A=Q\, \Lambda \, Q$ with $\Lambda=\mathrm{diag}(\mu_1,\mu_2,\dots,\mu_m)$. With $G_k:=G(\cdot;\mu_k)$ we define
\begin{equation}\label{eq:matsol}
\mathcal{S}(t):=Q\mathrm{diag}(G_1(t),G_2(t),\dots,G_m(t))Q\text{ and } \mathcal{C}(t):=\mathcal{S}'(t)+\beta\mathcal{S}(t).
\end{equation}
Equivalently, $\mathcal{S}(t)$ is \eqref{eq:taylorsingle} evaluated at $A$, see for example \cite{NH08},
\[\mathcal{S}(t)=e^{-\beta t/2}\sum^\infty_{j=0}\left(A+\frac{\beta^2}{4}I\right)^j\frac{t^{2j+1}}{(2j+1)!}.\]

\begin{prop}[Fundamental solution in all damping regimes]\label{prop:matexp}
Let $\beta\geqslant0$ and let $\mathcal{S}$ and $\mathcal{C}$ be given by \eqref{eq:matsol}. Then the solution of $\ddot{\vec{U}}+\beta\, \dot{\vec{U}}=A\, \vec{U}$ with $\vec{U}(0)=\vec{U}_0$ and $\dot{\vec{U}}(0)=\vec{V}_0$ is
\[\vec{U}(t)=\mathcal{C}(t)\, \vec{U}_0+\mathcal{S}(t)\, \vec{V}_0,\]
and for all $t\in\mathbb{R}$ we have
\begin{equation}\label{eq:matexp}
    e^{Mt}=
    \begin{bmatrix}
        \mathcal{C}(t) & \mathcal{S}(t) \\
        A\mathcal{S}(t) & \mathcal{S}'(t)
    \end{bmatrix}.
\end{equation}
    
\end{prop}
\begin{proof}
    Since $A=Q\, \Lambda \, Q$ the matrix $\vec{Y}=Q\, \vec{U}$ solves the diagonal system 
    \[\ddot{\vec{Y}}+\beta\, \dot{\vec{Y}}=\Lambda\, \vec{Y}\]
    with $\vec{Y}(0)=Q\, \vec{U}_0$ and $\dot{\vec{Y}}(0)=Q\, \vec{V}_0$. Its $k$-th row $y_k$ therefore solves the scalar equation
    \[y''_k+\beta\, y'_k=\mu_k\, y_k.\]
    By uniqueness and the scalar solutions $G_k$ and $G'_k+\beta G_k$ above we obtain
    \[y_k(t)=\left(G'_k(t)+\beta \, G_k(t)\right)y_k(0)+G_k(t)\, y'_k(0).\]
    In matrix form, 
    \[\vec{Y}(t)=\mathrm{diag}\left(G'_k(t)+\beta \, G_k(t)\right)Q\vec{U}_0+\mathrm{diag}\left(G_k(t)\right)Q\, \vec{V}_0,\]
    and multiplying by $Q$ gives $\vec{U}=\mathcal{C}\vec{U}_0+\mathcal{S}\vec{V}_0$. The matrix exponential $e^{Mt}$ is the fundamental matrix of the homogeneous first order system \eqref{eq:sdhcs_firstorder}, that is, it maps the initial state $\left(\vec{U}_0,\vec{V}_0\right)$ to the state $\left(\vec{U}(t),\dot{\vec{U}}(t)\right)$, so the first block row of \eqref{eq:matexp} is the formula just obtained. For the second it is $\dot{\vec{U}}=\mathcal{C}'\vec{U}_0+\mathcal{S}'\vec{V}_0$ and $\mathcal{C}'=\mathcal{S}''+\beta\, \mathcal{S}'=A\mathcal{S}$ since each $G_k$ satisfies $G''_k+\beta \, G'_k=\mu_k\, G_k.$
\end{proof}

\begin{cor}[Exponential decay]\label{cor:expdecay} 
Let $\beta>0$ and 
\[\delta_*\vcentcolon=-\mathrm{Re}\lambda^+_1=\frac{\beta}{2}-\sqrt{\max\{\mu_1+\beta^2/4,0\}}.\]
Then $\delta_*>0$ and for every $\delta\in(0,\delta_*)$ there is $c_\delta>0$ such that
\[\left|\left|e^{Mt}\right|\right|\leqslant c_\delta e^{-\delta t}\text{ for all }t\geqslant0.\]
\end{cor}
\begin{proof}
    By Proposition \ref{prop:esystem} the eigenvalues of $M$ are the $\lambda_k^\pm$ and since $\mu_1>\mu_2>\dots>\mu_m$ (Remark \ref{rem:damping}) the one with the largest real part is $\lambda_1^+$, with $\mathrm{Re}\,\lambda_1^+<0$ as $\mu_1<0$ and $\beta>0$. The bound then follows from \cite[Corollary 3.6]{GT12}, which holds for every square matrix, including the critically damped case in which $M$ is not diagonalisable.
\end{proof}

\begin{rem}\label{rem:extend} \
\begin{enumerate}[label=(\roman*)]
    \item If $\beta\notin\{\beta_1^*,\beta_2^*,\dots,\beta_m^*\}$, then $M$ is diagonalisable by Proposition \ref{prop:esystem}, and \eqref{eq:matexp} coincides with $P\mathrm{diag}(e^{\lambda^+_1t},e^{\lambda_1^-t},e^{\lambda^+_2t},e^{\lambda_2^-t},\dots,e^{\lambda^+_mt},e^{\lambda_m^-t})P^{-1}$, where the columns of $P$ are the eigenvectors $v_k^\pm$.
    \item The construction extends to the matrix $B$ of Proposition \ref{prop:blockdiaggeneral} and Corollary \ref{cor:blockdiagspecific}, including the critically damped case in which $B$ is not diagonalisable. Since $T$ and $D$ are simultaneously diagonalisable, they commute, the substitution $\vec{U}=e^{-Dt/2}\vec{W}$ transforms $\ddot{\vec{U}}+D\dot{\vec{U}}=T\vec{U}$ into $\ddot{\vec{W}}=\left(T+\frac{1}{4}D^2\right)\vec{W}$, and with
    \[\mathcal{S}(t):=e^{-Dt/2}\sum_{j=0}^\infty\left(T+\frac{1}{4}D^2\right)^j\frac{t^{2j+1}}{(2j+1)!}\text{ and }\mathcal{C}(t):=\mathcal{S}'(t)+D\mathcal{S}(t),\]
    which solves $\mathcal{S}''+D\mathcal{S}'=T\mathcal{S}$ with initial values $(0,I)$ and $(I,0)$ respectively. We then obtain as above
    \[e^{Bt}=
    \begin{bmatrix}
        \mathcal{C}(t) & \mathcal{S}(t) \\
        T\mathcal{S}(t) & \mathcal{S}'(t)
    \end{bmatrix}.
    \]
\end{enumerate}
\end{rem}

\section{Solutions to the semi-discrete hyperbolic curvature flow with boundary} \label{sec:solutions}

\subsection{Self-similar solutions}\label{sec:selfsimilar}
Self-similar solutions occupy a privileged position in the analysis of geometric evolution equations. They are the simplest non-trivial solutions, evolving by a fixed combination of rigid motions and a global rescaling, and their classification often exposes the geometric content of the underlying flow. For the smooth curve shortening flow in the plane, Abresch and Langer \cite{UA86} classified the homothetically shrinking closed immersed curves, while for non-compact curves Halldorsson \cite{HH12} gave a comprehensive classification including translators, rotators and dilators together with the geometry arising from their combinations. Altschuler, Altschuler, Angenent and Wu \cite{DA13} then extended this program to space curves.

In the semi-discrete setting the results have so far been more constrained. For closed polygons under the parabolic flow only scaling self-similar solutions exist \cite{JM25}, while the hyperbolic closed polygon flow also admits a family of pure rotations in the undamped case \cite{JM25_1}. These are consequences of the periodic boundary conditions and the resulting circulant structure. Prescribing boundary data replaces the circulant matrix by a tridiagonal Toeplitz one, and in doing so the self-similar solution landscape becomes richer; scaling, rotation and translation all occur, both individually and in combination. In \cite{JM24_1} the present authors constructed scaling, translating and rotating solutions of the parabolic flow with boundary by three independent arguments. The purpose of this section is to give a unified treatment for \eqref{eq:sdhcs}. 
\begin{defn}[Self-similar solution]
    Let $I\subset\mathbb{R}$ be an interval containing 0. A family of piecewise linear curves $\vec{X}(t)$ in $\mathbb{R}^p$, indexed by $t\in I$, is \emph{self-similar} if its members are related by
    \begin{equation}\label{eq:selfsimilardef}
        \vec{X}(t)=g(t)\vec{X}_0R(t)+\boldsymbol{1}\otimes h(t)
    \end{equation}
    for every $t$, where $g:I\to(0,\infty)$ is a twice differentiable \emph{scaling factor}, $R:I\to SO(p)$ is a twice differentiable \emph{rotation}, $h:I\to\mathbb{R}^p$ is a twice differentiable \emph{translation} and $\boldsymbol{1}\otimes h(t)$ denotes the $n\times p$ matrix with every row equal to $h(t)$. To ensure $\vec{X}(0)=\vec{X}_0$ we assume $g(0)=1,R(0)=I$ and $h(0)=0$.
\end{defn}

\begin{rem}[Self-similar boundary conditions]
    For the ansatz \eqref{eq:selfsimilardef} to be consistent with the boundary terms of \eqref{eq:sdhcs}, the prescribed boundary trajectories must themselves be self-similar, that is,
    \[f_1(t)=g(t)X_1(0)R(t)+h(t)\text{ and } f_n(t)=g(t)X_n(0)R(t)+h(t).\]
\end{rem}
Throughout this section we restrict our attention to plane curves, that is $p=2$. Any twice differentiable family $R:I\to SO(2)$ with $R(0)=I$ may be written $R(t)=e^{\theta(t)J}$ where
\[J=
\begin{bmatrix}
0 & -1 \\
1 & 0
\end{bmatrix}
\]
is the standard generator of $\mathfrak{so}(2)$ and $\theta$ is twice differentiable with $\theta(0)=0$. Writing $\omega\vcentcolon=\theta'$ for angular velocity we then have $R'(t)=R(t)\Omega(t)$ with $\Omega(t)\vcentcolon=\omega(t)J\in\mathfrak{so}(2)$. 
\begin{rem}
    The matrices $xI+yJ$ for $x,y\in\mathbb{R}$ form a commutative algebra isomorphic to $\mathbb{C}$ via $xI+yJ\mapsto x+iy$. Throughout this section we use this freely writing for example $M_0\leftrightarrow m_0$.
\end{rem}
We call the initial curve \emph{non-degenerate} if its interior vertices are not all coincident, a condition excluding only the configuration in which the entire interior collapses to a single point. Finally, we write throughout
\[\alpha \vcentcolon=e_1X_1(0)+e_{n-2}X_n(0)\]
for the $(n-2)\times p$ matrix whose only non-zero rows are the first and last, equal to the initial boundary positions.

\begin{prop}[Reduction to a constant matrix]\label{prop:selfsimconstant}
    Let $\vec{X}_0$ be a non-degenerate initial curve and let $g,R$ and $h$ be as in \eqref{eq:selfsimilardef}. Then \eqref{eq:selfsimilardef} is a solution of \eqref{eq:sdhcs} on $I$ with self-similar boundary conditions if and only if there are constants $M_0=p_0I+q_0J$ and $b\in\mathbb{R}^{1\times2}$ such that
    \begin{align}
        A\vec{U}_0-\vec{U}_0M_0=\mathbf{1}\otimes b-\alpha,\label{eq:selfsimilareq4}\\
        u''+\beta u'=m_0u,\quad u\vcentcolon=ge^{i\theta},\label{eq:selfsimilareq5}\\
        h''+\beta h'=gbR.\label{eq:selfsimilareq6}
    \end{align}
    The constants $M_0$ and $b$ are then unique.
\end{prop}
\begin{proof}
    Write $S(t)=g(t)R(t)$, then the interior of \eqref{eq:selfsimilardef} reads
    \[\vec{U}(t)=\vec{U}_0S(t)+\boldsymbol{1}\otimes h(t).\]
    Define
    \[M(t)\vcentcolon=\frac{g''(t)+\beta g'(t)}{g(t)}I+\left(\frac{2g'(t)}{g(t)}+\beta\right)\Omega(t)+\Omega'(t)+\Omega^2(t)\]
    and
    \[b(t)\vcentcolon=\frac{1}{g(t)}\left(h''(t)+\beta h'(t)\right)R^{-1}(t).\]
    Since all factors commute,
    \begin{multline*}
        S''(t)+\beta S'(t)= \\
        \left( (g''(t)+\beta g'(t))I+(2g'(t)+\beta g(t))\Omega(t)+g(t)(\Omega'(t)+\Omega^2(t))\right)R(t)=g(t)M(t)R(t),
    \end{multline*}
    and therefore
    \[\ddot{\vec{U}}(t)+\beta\dot{\vec{U}}(t) = g(t)\vec{U}_0M(t)R(t)+\boldsymbol{1}\otimes(h''(t)+\beta h'(t)).\]
    For the right hand side of \eqref{eq:sdhcs}, observe that the interior row sums of $A$ vanish whilst the first and last row sums equal $-1$, so $A(\boldsymbol{1}\otimes h(t))=-(e_1+e_{n-2})h(t)$. The self-similar boundary forcing is
    \[\vec{f}(t)=e_1f_1(t)+e_{n-2}f_n(t)=g(t)\alpha R(t)+(e_1+e_{n-2})h(t),\]
    thus the two $h$ contributions cancel exactly and
    \[A\vec{U}(t)+\vec{f}(t)=g(A\vec{U}_0+\alpha)R(t).\]
    Post multiplying by $R^{-1}$ and dividing by $g>0$ shows that \eqref{eq:selfsimilardef} solves \eqref{eq:sdhcs} if and only if
    \begin{equation}\label{eq:selfsimeq1}
        \vec{U}_0M(t)+\boldsymbol{1}\otimes b(t)=A\vec{U}_0+\alpha\text{ for all }t\in I.
    \end{equation}
    Suppose \eqref{eq:selfsimeq1} holds. The second term on the left has all rows equal and the right hand side is independent of $t$, so subtracting row $j$ from row $i$ of \eqref{eq:selfsimeq1} gives
    \[\left(\left(\vec{U}_0\right)_i-\left(\vec{U}_0\right)_j\right)M(t)=\left(A\vec{U}_0+\alpha\right)_i-\left(A\vec{U}_0+\alpha\right)_j\quad\text{for all }t\in I.\]
    By the non-degeneracy requirement there are $i,j$ for which $s=\left(\vec{U}_0\right)_i-\left(\vec{U}_0\right)_j\neq0$ and we write $c\vcentcolon=\left(A\vec U_0+\alpha\right)_i-\left(A\vec{U}_0+\alpha\right)_j$. Since $M(t)=\varphi_1(t)I+\varphi_2(t) J$ with
    \[\varphi_1(t)=\frac{g''(t)+\beta g'(t)}{g(t)}-\omega^2(t)\text{ and }\varphi_2(t)=\left(\frac{2g'(t)}{g(t)}+\beta\right)\omega(t)+\omega'(t),\]
  we have
  \[c=sM(t)=\varphi_1(t)s+\varphi_2(t)sJ.\]
  Since $J$ is a rotation by $\pi/2$, the vectors $s$ and $sJ$ are orthogonal with $|sJ|=|s|>0$. Taking the Euclidean inner product with $s$ and $sJ$ gives
  \[\varphi_1(t)=\frac{\langle c,s\rangle}{|s|^2}\text{ and }\varphi_2(t)=\frac{\langle c,sJ\rangle}{|s|^2},\]
  so $M(t)\equiv M_0=p_0I+q_0J$ is constant, with $p_0=\langle c,s\rangle/|s|^2$ and $q_0=\langle c,sJ\rangle/|s|^2$. Equation \eqref{eq:selfsimeq1} becomes
  \[\boldsymbol{1}\otimes b(t)=A\vec{U}_0+\alpha-\vec{U}_0M_0.\]
  Every row of the left hand side is $b(t)$ and the right hand side is independent of $t$, thus $b(t)\equiv b$ is constant and \eqref{eq:selfsimeq1} becomes \eqref{eq:selfsimilareq4}. Finally $S''+\beta S'=g M(t)R$ becomes $S''+\beta S'=M_0S$ which is \eqref{eq:selfsimilareq5} and the definition of $b(t)$ becomes $h''+\beta h'=gbR$ which is \eqref{eq:selfsimilareq6}.

  Conversely, we suppose \eqref{eq:selfsimilareq4}-\eqref{eq:selfsimilareq6} hold. Then \eqref{eq:selfsimilareq5} gives $gM(t)R=S''+\beta S'=M_0S=gM_0R$, so $M(t)=M_0$ and $\eqref{eq:selfsimilareq6}$ gives $b(t)=b$. Hence \eqref{eq:selfsimeq1} reduces to \eqref{eq:selfsimilareq4}, which holds. In particular $M_0=M(t)$ and $b=b(t)$ are determined by $g, R$ and $h$, thus they are unique.
\end{proof}

Proposition \ref{prop:selfsimconstant} separates the classification into a sequence of questions which the remainder of this subsection addresses.
\begin{lem}[Solution of the nonlinear coupled system]\label{lem:selfsimilarcoupled} 
    Let $p_0,q_0\in\mathbb{R}$ and write both $M_0=p_0I+q_0J$ and $m_0=p_0+iq_0$.
    \begin{enumerate}[label=(\roman*)]
        \item There exists $\gamma,\omega_0\in\mathbb{R}$ such that, for $\Omega_0=\omega_0J$,
        \begin{equation}\label{eq:selfsimeq3}
            M_0=(\gamma^2+\beta\gamma)I+(2\gamma+\beta)\Omega_0+\Omega_0^2 \mbox{.}
        \end{equation}
        \item The coupled system
        \[\frac{g''(t)+\beta g'(t)}{g(t)}-\omega^2(t)=p_0,\qquad \left(\frac{2g'(t)}{g(t)}+\beta\right)\omega(t)+\omega'(t)=q_0\]
        with $g(0)=1$ and initial data $z_0\vcentcolon=g'(0)+i\omega(0)\in\mathbb{C}$ has, for $(\gamma,\omega_0)\neq(-\beta/2,0)$, the solution
        \begin{equation}\label{eq:selfsimeq2}
            g(t)e^{i\theta(t)}=\frac{e^{(\gamma+i\omega_0)t}+\rho e^{-(\gamma+\beta+i\omega_0)t}}{1+\rho}
        \end{equation}
        where $\rho=(r_1-z_0)/(z_0-r_2)$ with $r_1=\gamma+i\omega_0$ and $r_2=-\gamma-\beta-i\omega_0$ if $z_0\neq r_2$ and in the case $z_0=r_2$ has the solution
        \[g(t)e^{i\theta(t)}=e^{r_2t}.\]
        In the case $(\gamma,\omega_0)=(-\beta/2,0)$, where the two exponents in \eqref{eq:selfsimeq2} coincide, the solution 
        \[g(t)e^{i\theta(t)}=(1+ct)e^{-\beta t/2}\]
        where $c=z_0+\beta/2$. In each case $g(t)$ and $\theta(t)$ are the modulus and argument of the right hand side, on the maximal interval $I_{z_0}$ containing zero on which the right hand side does not vanish. The rotation matrix is recovered by $R(t)=e^{\theta(t)J}$.
    \end{enumerate}
\end{lem}
\begin{proof}
    Consider the substitution
\[f(t)=\frac{g'(t)}{g(t)}\implies f'(t) = \frac{g''(t)}{g(t)}-f^2(t).\]
Then
\begin{align*}
	f'(t)+f^2(t)+\beta f(t)-\omega^2(t)&=p_0,\\
	\left(2f(t)+\beta\right)\omega(t)+\omega'(t)&=q_0.
\end{align*}
Now consider $z(t)=f(t)+i\omega(t)$ then
\begin{align*}
	z'(t)&=f'(t)+i\omega'(t),\\
	z^2(t)&=\left(f(t)+i\omega(t)\right)^2=\left(f^2(t)-\omega^2(t)\right)+ i\left(2f(t)\omega(t)\right),\\
	\beta z(t) &= \beta f(t)+i\beta\omega(t).
\end{align*}
So
\begin{align*}
	z'(t)+z^2(t)+\beta z(t) &= \left(f'(t)+f^2(t)+\beta f(t) - \omega^2(t)\right) + i\left(\omega'(t)+2f(t)\omega(t)+\beta\omega(t)\right)\\
				&=p_0+iq_0=m_0.
\end{align*}
We now have a Riccati equation, as such we utilise the substitution
\[z(t)=\frac{u'(t)}{u(t)}\implies z'(t)=\frac{u''(t)}{u(t)}-z^2(t).\]
Now the equation becomes linear in $u$,
\begin{equation*}
	\frac{u''(t)}{u(t)} + \beta\frac{u'(t)}{u(t)}=m_0
	\implies u''(t)+\beta u'(t)-m_0\, u(t)=0.
\end{equation*}
Using the ansatz $u(t)=e^{rt}$ we obtain
\[m_0=r^2+\beta r.\]
Setting $r=\gamma+i\omega_0$ so
\begin{align*}
	m_0&=(\gamma+i\omega_0)^2+\beta(\gamma+i\omega_0)
	   =\gamma^2+\beta\gamma-\omega_0^2+i(2\gamma\omega_0+\beta\omega_0)\\
	&\implies M_0=(\gamma^2+\beta\gamma)I+(2\gamma +\beta)\Omega_0+\Omega_0^2.
\end{align*}
Thus we have (i). Now returning to solving the ordinary differential equation
\begin{equation*}
	r = \frac{-\beta\pm\sqrt{\beta^2+4m_0}}{2}
	  \implies r_1,r_2=\gamma+i\omega_0,-\gamma-\beta-i\omega_0.
\end{equation*}
Thus we can write 
\begin{align*}
	u(t) &= c_1e^{(\gamma+i\omega_0)t}+c_2e^{-(\gamma+\beta+i\omega_0)t}\\
	\implies z(t)&=\frac{u'(t)}{u(t)}=\frac{c_1(\gamma+i\omega_0)e^{(\gamma+i\omega_0)t}-c_2(\gamma+\beta+i\omega_0)e^{-(\gamma+\beta+i\omega_0)t}}{c_1e^{(\gamma+i\omega_0)t}+c_2e^{-(\gamma+\beta+i\omega_0)t}}.
\end{align*}
If $c_1=0$ then $u(t)/u(0)=e^{r_2t}$. Otherwise, dividing through by $c_1e^{(\gamma+i\omega_0)t}$ and setting $\rho=c_2/c_1$ we obtain
\[z(t) =\frac{(\gamma+i\omega_0) - \rho(\gamma+\beta+i\omega_0)e^{-(2\gamma+\beta+i2\omega_0)t}}{1+\rho e^{-(2\gamma+\beta+i2\omega_0)t}}. \]
Recall $z=u'/u$ and $z=g'/g+i\omega$ then
\[\int_0^t z(s)\;\mathrm{d}s=\ln u(t) - \ln u(0).\]
The left side separates as
\[\int_0^tf(s)\;\mathrm{d}s + i\int_0^t\omega(s)\;\mathrm{d}s=\ln g(t)+i\theta(t) \implies \ln g(t)+i\theta(t)=\ln \frac{u(t)}{u(0)}.\]
Exponentiating
\[g(t)e^{i\theta(t)}=\frac{u(t)}{u(0)}=\frac{c_1e^{(\gamma+i\omega_0)t}+c_2e^{-(\gamma+\beta+i\omega_0)t}}{c_1+c_2}=\frac{e^{(\gamma+i\omega_0)t}+\rho e^{-(\gamma+\beta+i\omega_0)t}}{1+\rho}.\]
In the case $(\gamma,\omega_0)=(-\beta/2,0)$ we have $\beta^2+4m_0=0$ and the repeated root $r_1,r_2=-\beta/2$, so
\[u(t)=(c_1+c_2t)e^{-\beta t/2},\]
and the same integration gives
\[g(t)e^{i\theta(t)}=\frac{u(t)}{u(0)}=(1+ct)e^{-\beta t/2}\text{ for }c=\frac{c_2}{c_1}.\]
In each case $u=u(0)ge^{i\theta}$, so $u(0)\neq0$ (i.e. $\rho\neq-1$) and $u'(0)/u(0)=z(0)=z_0$. Further $u$ does not vanish on an interval on which $g>0$ so every such interval lies within $I_{z_0}$. Conversely, on $I_{z_0}$ the right hand side has a positive modulus and continuous argument $\theta$ and reversing the substitutions $z=u'/u$ and $f=g'/g$ shows that $g$ and $\omega=\theta'$ solve the coupled system.
\end{proof}

With Lemma \ref{lem:selfsimilarcoupled} in hand, we look to the configurations that are admissible. By Proposition \ref{prop:selfsimconstant}, the translation enters the equation \eqref{eq:selfsimilareq4} for the profile only through the constant vector $b$. We first treat $b=0$, that is, scaling-rotations centered at the origin, possibly followed by a free translation (Proposition \ref{prop:selfsimexistence}, Corollary \ref{cor:classification} and Lemma \ref{lem:appendtranslation}). We then show that for $m_0\neq0$ the general case reduces to this one by moving the centre (Lemma \ref{lem:selfsimilarmovecentre}), whereas for $m_0=0$ the vector $b$ is an intrinsic parameter of the motion (Proposition \ref{prop:discreteenergy} and Theorem \ref{thm:selfsimilar}).

\begin{prop}[Existence and uniqueness of a configuration]\label{prop:selfsimexistence}
     Let $(\gamma,\omega_0,\beta)$ be motion parameters with associated matrix $M_0$ as in \eqref{eq:selfsimeq3}, let  $b\in\mathbb{R}^{1\times2}$ and let $X_1(0),X_n(0)\in\mathbb{R}^2$ be a boundary configuration. Then the Sylvester equation \eqref{eq:selfsimilareq4} has a unique solution if and only if $m_0\neq\mu_k$ for every $k$, that is, unless $M_0=\mu_kI$ for some $k$. If $M_0=\mu_kI$, then \eqref{eq:selfsimilareq4} is solvable if and only if
     \begin{equation}\label{eq:sylvestereval}
         \langle u_k,\boldsymbol{1}\rangle b=(u_k)_1X_1(0)+(u_k)_{n-2}X_n(0),
     \end{equation}
     and its solutions are then $\vec{U}_0+u_kc,c\in\mathbb{R}^{1\times2}$ for any one solution $\vec{U}_0$.
\end{prop}
\begin{proof}
    By the classical Sylvester Theorem (see for example \cite[Theorem 2.4.4.1]{RH12}), \eqref{eq:selfsimilareq4} has a unique solution if and only if $A$ and $M_0$ have no eigenvalue in common. The eigenvalues of $M_0=p_0I+q_0J$ are $m_0$ and $\overline{m_0}$, while those of $A$ are real numbers $\mu_k$, so a common eigenvalue exists if and only if $m_0=\mu_k$ for some $k$, that is, $M_0=\mu_kI$. In this case \eqref{eq:selfsimilareq4} becomes
    \[(A-\mu_kI)\vec{U}_0=\boldsymbol{1}\otimes b-\alpha.\]
    Since $A-\mu_kI$ is symmetric with nullspace spanned by $u_k$, this is solvable if and only if both columns of the right hand side are orthogonal to $u_k$, that is
    \[\langle u_k,\boldsymbol{1}\rangle b-(u_k)_1X_1(0)-(u_k)_{n-2}X_n(0)=0,\]
    which is \eqref{eq:sylvestereval}. The solutions are then unique up to adding $u_kc$.
\end{proof}

\begin{cor}[Profile types]\label{cor:classification}
    The profile $\vec{U}_0$ of a scaling-rotation, that is, the solution of \eqref{eq:selfsimilareq4} with $b=0$ from Proposition \ref{prop:selfsimexistence}, is classified by the scaling-rotation pair $(\gamma,\omega_0)$ into three non-trivial types:
    \begin{enumerate}[label=(\roman*)]
        \item Pure scaling $(\gamma\neq0,\omega_0=0)$. A unique profile exists if and only if $\gamma^2+\beta\gamma\neq\mu_k$ for every $k$.
        \item Pure rotation $(\gamma=0,\omega_0\neq0)$. A unique profile always exists when $\beta>0$. When $\beta=0$ a unique profile exists if and only if $\omega_0^2\neq-\mu_k$ for every $k$.
        \item Scaling with rotation $(\gamma\neq0,\omega_0\neq0)$. A unique profile exists when $2\gamma+\beta\neq0$. When $2\gamma+\beta=0$ a unique profile exists if and only if $-\beta^2/4-\omega_0^2\neq\
        \mu_k$ for every $k$.
    \end{enumerate}
\end{cor}
\begin{rem}
    Note the labels pure scaling and pure rotation describe the motion where $\rho=0$ in \eqref{eq:selfsimeq2}. For other initial data the same $M_0$ generates, in general, a combination of scaling and rotation. The labels are chosen to coincide with the naming of $\gamma$ and $\omega_0$ being the scaling and rotation parameters respectively.
\end{rem}
\begin{proof}
By Proposition \ref{prop:selfsimexistence} a unique profile exists if and only if $m_0\neq\mu_k$ for every $k$, which always holds when the imaginary part $(2\gamma+\beta)\omega_0$ of $m_0$ is non-zero, since the $\mu_k$ are real.
\begin{enumerate}[label=(\roman*)]
    \item With $\omega_0=0$ we have $M_0=(\gamma^2+\beta\gamma)I$ and \eqref{eq:selfsimilareq4} with $b=0$ reduces to
    \[(A-(\gamma^2+\beta\gamma)I)\vec{U}_0=-\alpha\]
    which is uniquely solvable if and only if $\gamma^2+\beta\gamma\neq\mu_k$ for every $k$.
    \item With $\gamma=0$ we have $m_0=-\omega_0^2+i\beta\omega_0$, which is non-real when $\beta\omega_0\neq0$. For $\beta=0$ it equals $-\omega_0^2$, which equals $\mu_k$ when $\omega_0^2=-\mu_k$ for some $k$.
    \item The imaginary part of $m_0$ is $(2\gamma+\beta)\omega_0$ so $m_0\neq\mu_k$ when $2\gamma+\beta\neq0$, however when $2\gamma+\beta=0$ we have that $\gamma=-\beta/2$, so $m_0=-\beta^2/4-\omega_0^2$ and a unique solution exists if and only if $-\beta^2/4-\omega_0^2\neq\mu_k$ for every $k$.
\end{enumerate}
\end{proof}

\begin{lem}[Appending a translation]\label{lem:appendtranslation}
    Let $\vec{X}(t)=g(t)\vec{X}_0R(t)$ be a scaling-rotation solution of \eqref{eq:sdhcs}, that is, a solution as in Proposition \ref{prop:selfsimconstant} with $h\equiv0$. Then $\vec{X}(t)+\boldsymbol{1}\otimes h(t)$ with the correspondingly translated boundary condition is also a self-similar solution if and only if
    \[h''(t)+\beta h'(t)=0.\]
    With $h(0)=0$ the non-trivial solutions are
    \[h(t)=c(1-e^{-\beta t})\text{ for }\beta>0\text{ and }h(t)=ct\text{ for }\beta=0.\]
\end{lem}
\begin{proof}
Write $\vec{f}(t)=e_1f_1(t)+e_{n-2}f_n(t)=g(t)\alpha R(t)$ for the boundary forcing of the scaling-rotation, so that
\[\ddot{\vec{U}}+\beta\dot{\vec{U}}=A\vec{U}+\vec{f}(t).\]
The translated curve $\vec{X}(t)+\boldsymbol{1}\otimes h(t)$ has interior $\vec{U}(t)+\boldsymbol{1}\otimes h(t)$ and boundary trajectories $f_1(t)+h(t)$ and $f_n(t)+h(t)$, hence boundary forcing $\vec{f}(t)+(e_1+e_{n-2})h(t)$. Substituting into \eqref{eq:sdhcs}, the left hand side is
\[\frac{\mathrm{d}^2}{\mathrm{d}t^2}\left(\vec{U}(t)+\boldsymbol{1}\otimes h(t)\right) + \beta\frac{\mathrm{d}}{\mathrm{d}t}\left(\vec{U}(t)+\boldsymbol{1}\otimes h(t)\right)=\ddot{\vec{U}}(t)+\beta\dot{\vec{U}}+\boldsymbol{1}\otimes(h''(t)+\beta h'(t)),\]
while, since the interior row sums of $A$ vanish and the first and last equal $-1$, the right hand side is
\begin{multline*}
       A(\vec{U}+\boldsymbol{1}\otimes h(t))+\vec{f}(t)+(e_1+e_{n-2})h(t)=A\vec{U}-(e_1+e_{n-2})h(t) \\ +\vec{f}(t)+(e_1+e_{n-2})h(t)=A\vec{U}+\vec{f}(t). 
\end{multline*}
Subtracting the equation satisfied by $\vec{U}$, the translated curve solves \eqref{eq:sdhcs} if and only if
\[\boldsymbol{1}\otimes(h''(t)+\beta h'(t))=0 \implies h''(t)+\beta h'(t)=0.\]
The characteristic roots of this equation are $0$ and $-\beta$, so
\[h(t)=c_0+c_1e^{-\beta t}\text{ for }\beta>0\text{ and }h(t)=c_0+c_1t\text{ for }\beta=0,\]
imposing $h(0)=0$ provides $c_1=-c_0$ in the $\beta>0$ case, resulting in
\[h(t)=c\left(1-e^{-\beta t}\right)\text{ with }c=c_0\]
and in the $\beta=0$ case the initial condition gives $c_0=0$, so that, relabelling $c=c_1,$
\[h(t)=c\, t.\]
\end{proof}

\begin{lem}[Change of centre]\label{lem:selfsimilarmovecentre}
    Let $\vec{X}(t)$ be a self-similar solution as in Proposition \ref{prop:selfsimconstant} with $m_0\neq0$ and let $d \vcentcolon=-b\, M_0^{-1}$. Then $\vec{X}(t)-\boldsymbol{1}\otimes d$ is a self-similar solution with $b=0$ and
    \begin{equation}\label{eq:changeofcentre}
        \vec{X}(t)=\boldsymbol{1}\otimes d+\left(\vec{X}_0-\boldsymbol{1}\otimes d\right)g(t)R(t)+\boldsymbol{1}\otimes\tilde{h}(t),\quad \tilde{h}\vcentcolon=h-d(I-gR),
    \end{equation}
    that is, $\vec{X}$ is a scaling-rotation about the point $d$ followed by the translation $\tilde{h}$. Conversely, for every $d\in\mathbb{R}^{1\times 2}$, adding $\boldsymbol{1}\otimes d$ to a self-similar solution with $b=0$ gives a self-similar solution with $b=-dM_0$.
\end{lem}
\begin{proof}
Translating a curve its boundary data  by a constant vector preserves \eqref{eq:sdhcs}, since $A(\boldsymbol{1}\otimes d)=-(e_1+e_{n-2})\otimes d$ cancels the change in $\vec{f}$. Thus $\vec{X}(t)-\boldsymbol{1}\otimes d$ is a solution and since $(\boldsymbol{1}\otimes d)gR=\boldsymbol{1}\otimes(dgR)$,
\[\vec{X}(t)-\boldsymbol{1}\otimes d=g\vec{X}_0R+\boldsymbol{1}\otimes(h-d)=\left(\vec{X}_0-\boldsymbol{1}\otimes d\right)gR+\boldsymbol{1}\otimes\tilde{h},\]
which is \eqref{eq:changeofcentre}. This solution is self-similar with the same $g$ and $R$ and with translation $\tilde{h}$, where $\tilde{h}(0)=0.$ By \eqref{eq:selfsimilareq6} and \eqref{eq:selfsimilareq5}, which reads $(gR)''+\beta(gR)'=M_0\, g\, R$,
\[\tilde{h}''+\beta\tilde{h}'=h''+\beta h'+d((gR)''+\beta(gR)')=g(b+d\, M_0)R=0,\]
so by \eqref{eq:selfsimilareq6} it has $b=0$. On the other hand, if $\vec{Y}(t)$ is a self-similar solution with $b=0$ and translation $\tilde{h}$, the same computation shows that $\vec{Y}(t)+\boldsymbol{1}\otimes d$ is a self-similar solution with translation $h=\tilde{h}+d(I-gR)$ and $h''+\beta h'=-d((gR)''+\beta(gR)')=g(-d\, M_0)R$, so it has $b=-d\, M_0$.
\end{proof}

\begin{prop}[Pure translation]\label{prop:puretranslation}
    Let $\beta\geqslant0$. Given any vectors $v,h_1\in\mathbb{R}^2$ and endpoints $X_1(0),X_n(0)\in\mathbb{R}^2$, there is a unique solution of \eqref{eq:sdhcs} evolving by pure translation with initial velocity $h_1$ i.e.  $\vec{X}(t)=\vec{X}_0+\boldsymbol{1}\otimes h(t)$ with $h(0)=0,h'(0)=h_1$, whose interior profile has constant second difference $v$. The interior profile is the unique solution of
    \[A\vec{U}_0=\boldsymbol{1}\otimes v-\alpha,\]
    and the translation is $h(t)=h_1\psi_1(t)+v\psi_2(t)$, where
    \begin{equation}\label{eq:transpsi}
    \psi_1(t)\vcentcolon=\frac{1-e^{-\beta t}}{\beta},\quad\psi_2(t)\vcentcolon=\frac{\beta t-1+e^{-\beta t}}{\beta^2}
    \end{equation}
    for $\beta>0$ and $\psi_1(t)=t,\psi_2(t)=t^2/2$ for $\beta=0.$ Every pure translation solution is of this form.
\end{prop}
\begin{proof}
    A translating piecewise linear curve has the form $\vec{X}(t)=\vec{X}_0+\boldsymbol{1}\otimes h(t)$ where $\boldsymbol{1}\otimes h(t)$ is the $n\times p$ matrix with every row equal to some time-dependent vector $h(t)\in\mathbb{R}^p$ satisfying $h(0)=0$. In particular, every point, including the boundary translates by the same $h(t)$, so $X_1(t) = X_1(0) + h(t)$ and $X_n(t)=X_n(0) + h(t)$.

	We substitute directly into the pointwise evolution equation, for each interior vertex $i=2,3,\dots,n-1$
	\[\ddot{X}_i + \beta\dot{X}_i=X_{i-1}-2X_i + X_{i+1}.\]
	With $X_i(t)=X_i(0)+h(t)$ for all $i$, the left hand side becomes $\ddot{h}+\beta\dot{h}$. The right hand side becomes
	\[\left(X_{i-1}(0)+h\right) - 2\left(X_i(0)+h\right) + \left(X_{i+1}(0)+h\right) = X_{i-1}(0)-2X_i(0)+X_{i+1}(0), \]
	since $h(t)$ terms cancel exactly. Hence
	\[\ddot{h} + \beta\dot{h}=X_{i-1}(0)-2X_i(0)+X_{i+1}(0)\text{ for all }i=2,3,\dots,n-1.\]
	For this to yield the same vector $v$ for all interior vertices $i$ we require,
	\[X_{i-1}(0)-2X_i(0)+X_{i+1}(0)=v\text{ for }i=2,3\dots,n-1.\]
	In matrix form this is precisely 
	\begin{equation}\label{eq:puretrans}
		A\vec{U}_0 + e_1X_1(0) + e_{n-2}X_n(0) = \boldsymbol{1}\otimes v.
	\end{equation}
	Since $A$ is invertible, given any $v,X_1(0)$ and $X_n(0)$ \eqref{eq:puretrans} has a unique solution.

	The translation $h(t)$ then satisfies $h''+\beta h'=v$ with $h(0)=0$ and $h'(0)=h_1$. Since $\psi_1''+\beta\psi_1'=0,\psi''_2+\beta\psi'_2=1,\psi_1(0)=\psi_2(0)=0,\psi'_1(0)=1$ and $\psi_2'(0)=0$, its unique solution is 
    \[h=h_1\psi_1+v\psi_2.\]
\end{proof}

\begin{thm}[Classification of self-similar solutions]\label{thm:selfsimilar}
Let $\vec{X}(t),t\in I$ be a self-similar solution of \eqref{eq:sdhcs} with self-similar boundary conditions and non-degenerate initial curve, let $M_0\leftrightarrow m_0$ and $b$ be as in Proposition \ref{prop:selfsimconstant} and let $z_0=g'(0)+i\omega(0)$. Then $\vec{X}$ belongs to exactly one of the following classes:
\begin{enumerate}[label = (\roman*)]
    \item $m_0\neq0$. Then $\vec{X}$ is a scaling-rotation about the centre $d=-b\, M_0^{-1}$ followed by a translation,
    \[\vec{X}(t)=\boldsymbol{1}\otimes d+\left(\vec{X}_0-\boldsymbol{1}\otimes d\right)g(t)R(t)+\boldsymbol{1}\otimes\tilde{h}(t),\]
    where $g(t)e^{i\, \theta(t)}$ is given by Lemma \ref{lem:selfsimilarcoupled}, $\tilde{h}$ is given by Lemma \ref{lem:appendtranslation} and the profile of $\vec{X}_0-\boldsymbol{1}\otimes d$ solves \eqref{eq:selfsimilareq4} with $b=0$ for its own boundary data as classified by Corollary \ref{cor:classification}.
    \item $m_0=0$. Then, on the maximal interval containing zero on which $1+z_0\psi_1\neq0,$
    \begin{gather*}
        \vec{X}(t)=g(t)\vec{X}_0R(t)+\boldsymbol{1}\otimes h(t),\\
        g(t)e^{i\theta(t)}=1+z_0\psi_1(t),\\
        h(t)=h'(0)\psi_1(t)+b\psi_2(t)+b\Sigma_0\psi_3(t),
    \end{gather*}
    where $\Sigma_0\leftrightarrow z_0$ and $\psi_1,\psi_2$ are given by \eqref{eq:transpsi} and
    \[\psi_3(t)\vcentcolon=\frac{\beta t\left(1+e^{-\beta t}\right)-2\left(1-e^{-\beta t}\right)}{\beta^3}\]
    for $\beta>0$ and $\psi_3(t)=t^3/6$ for $\beta=0$ and the profile is the unique solution of $A\vec{U}_0=\boldsymbol{1}\otimes b-\alpha$.
\end{enumerate}
\end{thm}
\begin{proof} \
By Proposition \ref{prop:selfsimconstant} every self-similar solution determines $m_0$ and $b$, thus the two classes are mutually exclusive and exhaustive.
\begin{enumerate}[label=(\roman*)]
    \item  If $m_0\neq0$, Lemma \ref{lem:selfsimilarmovecentre} gives the representation in (i), with $g(t)e^{i\theta(t)}$ given by Lemma \ref{lem:selfsimilarcoupled}, $\tilde{h}$ by Lemma \ref{lem:appendtranslation} and the profile by Proposition \ref{prop:selfsimexistence}.
    \item  If $m_0=0$ then \eqref{eq:selfsimilareq5} reads as $u''+\beta u'=0$ with $u(0)=1$ and $u'(0)=z_0$, thus $u=1+z_0\psi_1$. Then \eqref{eq:selfsimilareq4} reads as $A\vec{U}_0=\boldsymbol{1}\otimes b-\alpha$ which is uniquely solvable as every $\mu_k<0$ and \eqref{eq:selfsimilareq6} is $h''+\beta h'=b(I+\Sigma_0\psi_1)$ with $h(0)=0$, whose solution is the stated one since, in addition to the properties of $\psi_1,\psi_2$ used in the proof of Proposition \ref{prop:puretranslation}, $\psi_3''+\beta \psi_3'=\psi_1$ with $\psi_3(0)=\psi_3'(0)=0.$
\end{enumerate}
\end{proof}

Figure \ref{fig:selfsimilar} shows a selection of self-similar solutions constructed via Theorem \ref{thm:selfsimilar}.
\begin{figure}[htb]
    \centering
    \begin{subfigure}[t]{0.32\textwidth}
        \begin{minipage}[t][4.2cm][t]{\linewidth}
            \centering
            \includegraphics[width=\linewidth,height=4.2cm,keepaspectratio]{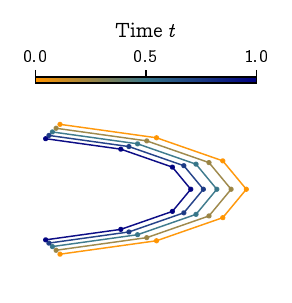}
        \end{minipage}
        \caption{Scaling.}
    \end{subfigure}
    \hfill
    \begin{subfigure}[t]{0.32\textwidth}
        \begin{minipage}[t][4.2cm][t]{\linewidth}
            \centering
            \includegraphics[width=\linewidth,height=4.2cm,keepaspectratio]{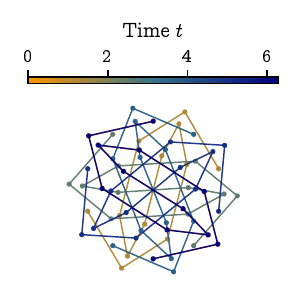}
        \end{minipage}
        \caption{Rotating.}
    \end{subfigure}
    \hfill
    \begin{subfigure}[t]{0.32\textwidth}
        \begin{minipage}[t][4.2cm][t]{\linewidth}
            \centering
            \includegraphics[width=\linewidth,height=4.2cm,keepaspectratio]{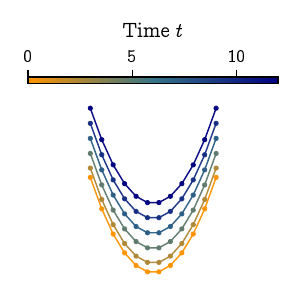}
        \end{minipage}
        \caption{Translating.}
    \end{subfigure}

    \begin{subfigure}[t]{0.32\textwidth}
        \begin{minipage}[t][4.2cm][t]{\linewidth}
            \vspace{0pt}
            \centering
            \includegraphics[width=\linewidth,height=4.2cm,keepaspectratio]{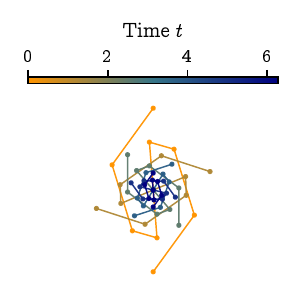}
        \end{minipage}
        \caption{Scaling and rotating.}
    \end{subfigure}\hfill
    \begin{subfigure}[t]{0.32\textwidth}
        \begin{minipage}[t][4.2cm][t]{\linewidth}
            \vspace{3.8pt}
            \centering
            \includegraphics[width=\linewidth,height=4.2cm,keepaspectratio]{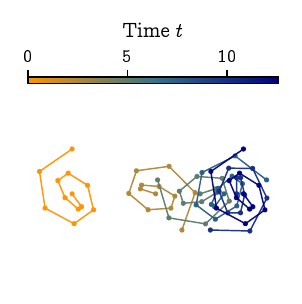}
        \end{minipage}
        \caption{Rotating and translation.}
    \end{subfigure}\hfill
    \begin{subfigure}[t]{0.32\textwidth}
        \begin{minipage}[t][4.2cm][t]{\linewidth}%
            \vspace{3.8pt}
            \centering
            \includegraphics[width=\linewidth,height=4.2cm,keepaspectratio]{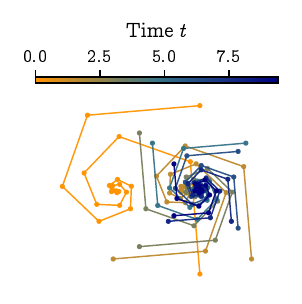}
        \end{minipage}
        \caption{Scaling, rotating and translating.}
    \end{subfigure}

    \caption{A selection of self-similar solutions. Time steps are shown
    superimposed where the colour change from yellow through teal to navy blue indicates the progression of time $t$. Parameters of the curves are given in Appendix \ref{sec:app}.}
    \label{fig:selfsimilar}
\end{figure}

%
\subsection{Time periodic solutions}\label{sec:timeperiodic}
Similar to self-similar solutions, time periodic solutions occupy an important place in the theory of evolutionary equations. This is because they are non-trivial recurrent solutions which, depending on the systems properties, may either persist indefinitely under conservative dynamics or arise only as a forced response under dissipative dynamics. For the closed polygon analogues of \eqref{eq:sdhcs} studied by the first author and Meyer in \cite{JM25_1}, time periodic solutions arise naturally in the undamped setting. In their notation, the case $\beta=0$ admits a family of solutions of the form
\[\vec{X}(t)=(\cos t)\vec{Y}+(\sin t)\vec{Z}\]
which they identify as \emph{breathers}, periodic solutions that oscillate between two polygonal states. The single frequency structure is a consequence of the scalar operator $-I$, where every mode \emph{sees} the same eigenvalue. For the polyharmonic flows of order $\ell\geqslant1$ in \cite{JM25_1}, the corresponding operator has distinct eigenvalues and each eigenmode oscillates at its own frequency.

For the open piecewise linear curves studied here each eigenmode of $A$ still oscillates between two states, but the boundary destroys the single frequency structure of the breathers present in the closed case. The only two natural routes to time periodic behaviour for the open flow are therefore (i) to set $\beta=0$, in which case the standing wave modes of the homogeneous equation oscillate at the natural frequencies $\omega_k$, or (ii) to drive the flow with periodic boundary forcing, in which case the dissipation cannot remove the periodic input and a unique periodic orbit is generated. The corresponding two cases are the subject of Propositions \ref{prop:tperiodicundamped} and \ref{prop:tperiodicdamped} below.
\subsubsection{The Undamped Case $(\beta=0)$}
We begin with the undamped case $\beta=0$, which is the closest in spirit to the breather solutions of \cite{JM25_1}. Here every homogeneous solution decomposes into a superposition of standing waves, each oscillating at the frequency determined by the corresponding eigenvalue of $A$.
\begin{prop}[Undamped standing waves]\label{prop:tperiodicundamped}
	Consider the semi-discrete hyperbolic flow \eqref{eq:sdhcs} with $\beta = 0$ and fixed boundary conditions and let $\vec{U}_\infty=-A^{-1}\alpha$ be the straight line interpolant. Let $\{(\mu_k, \phi_k)\}_{k=1}^{n-2}$ be the eigenpairs of the matrix $A$ with orthonormal eigenvectors $\phi_k$. The general solution $\vec{U}(t)$ is given by
\begin{equation}
	U_i(t) = \left(\vec{U}_\infty\right)_i + \sum_{k=1}^{n-2} \phi_{k,i} \left( C_k \cos(\omega_k t) + D_k \sin(\omega_k t) \right),
\end{equation}
	where $\omega_k=\sqrt{-\mu_k}$ are the natural frequencies, $\phi_{k,i}$ denotes the $i$-th component of the $k$-th eigenvector and $C_k,D_k\in\mathbb{R}^{1\times p}$ are constant vectors determined by the initial configuration $\vec{U}(0)$ and $\mathrm{d}\vec{U}/\mathrm{d}t|_{t=0}$. Moreover, $\vec{U}$ is periodic if and only if $\omega_k/\omega_l\in\mathbb{Q}$ for all $k,l$ with $(C_k,D_k)\neq0$ and $(C_l,D_l)\neq(0,0).$
\end{prop}
\begin{proof}
With $\beta=0$ and fixed boundaries, $\vec{Z}\vcentcolon=\vec{U}-\vec{U}_\infty$ satisfies
	\[\ddot{\vec{Z}}=A\, \vec{Z}.\]
    Taking $\phi_k$ to be the columns of matrix $Q$ of Section \ref{sec:mat} we have $Q=Q^T=Q^{-1}$ and
    \[A=Q\, \Lambda \, Q.\]
	Setting $\vec{Z}=Qa(t)$ so that $a=Q\vec{Z}$ has rows $a_k=\phi_k^\top\vec{Z}$, we obtain
	\[\ddot{a}_k=\mu_ka_k\text{ for each }k.\]
	Since $\mu_k<0$ we set $\omega_k=\sqrt{-\mu_k}$ so that
	\[\ddot{a}_k=-\omega_k^2a_k\implies a_k(t)=C_k\cos(\omega_kt)+D_k\sin(\omega_kt).\]
	The constants are then determined by the initial data such that
	\[C_k=\phi_k^\top\left(\vec{U}-\vec{U}_\infty\right)(0)\text{ and }D_k=\frac{1}{\omega_k}\phi_k^\top\dot{\vec{U}}(0).\]
    Let $K$ be the set of $k$ with $(C_k,D_k)\neq(0,0)$. If $\vec{U}$ is $T$-periodic, then so is each $a_k$, so $\omega_kT\in2\pi\mathbb{Z}$ for every $k\in K$ and hence $\omega_k/\omega_l\in\mathbb{Q}$ for all $k,l\in K$. Now suppose these ratios are rational. Fix $l\in K$ and write $\omega_k/\omega_l=p_k/q$ for $k\in K$ with integers $p_k$ and common denominator $q$. Then $T=2\pi q/\omega_l$ satisfies $\omega_kT=2\pi p_k$ for every $k\in K$, so each $a_k$ and hence $\vec{U}$ is $T$-periodic.
\end{proof}

Figure \ref{fig:undamped_time_periodic} shows an example standing wave solution of Proposition \ref{prop:tperiodicundamped}.

\begin{figure}[htb]
	\centering
	\includegraphics[width=0.4\textwidth]{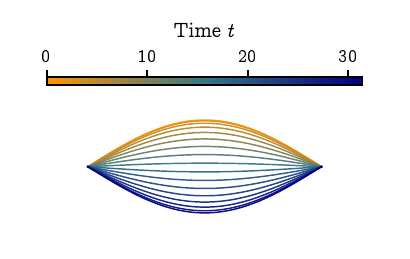}
	\caption[Example of an undamped time periodic solution.]{Example of an undamped time periodic solution. Time steps are shown superimposed where the colour change from yellow through teal to navy blue indicates the progression of time $t$. Parameters of the curve are given in Appendix \ref{sec:app}.}
	\label{fig:undamped_time_periodic}
\end{figure}

\subsubsection{The Damped Case $(\beta>0)$}
When $\beta>0$, Corollary \ref{cor:expdecay} guarantees that all homogeneous solutions decay exponentially to zero, thus no intrinsic breathers exist. The closed polygon analogue of this observation appears in \cite{JM25_1} as the asymptotic breathing phenomenon. The rescaled solutions $e^{\beta t/2}\vec{X}(t)$ oscillate between two states even though $\vec{X}(t)$ itself decays to a point. For the open curve the same rescaling makes any single underdamped mode periodic, but since different modes oscillate at different frequencies and overdamped modes do not oscillate, a general solution does not become periodic after rescaling. In the absence of forcing, solutions damp out without a residual oscillating component.

Periodic behaviour can however be sustained if the system is driven by periodic boundary conditions, an option unavailable to the closed polygons as they have no boundary. 
\begin{prop}[Forced periodic]\label{prop:tperiodicdamped}
	Consider the semi-discrete hyperbolic flow \eqref{eq:sdhcs} with damping $\beta>0$. Suppose the boundary forcing functions $f_1,f_n$ are locally integrable and $T$-periodic, that is $f_1(t+T)=f_1(t)$ and $f_n(t+T)=f_n(t)$. Then:
	\begin{enumerate}[label=(\roman*)]
		\item There exists a unique $T$-periodic solution $\vec{U}^*(t)$.
		\item This solution is globally asymptotically stable. For any other solution $\vec{U}(t)$,
			\[\lim_{t\to\infty}||\vec{U}(t)-\vec{U}^*(t)||=0.\]
	\end{enumerate}
\end{prop}
\begin{proof} 
	\leavevmode
	\begin{enumerate}[label=(\roman*)]
		\item We re-write the system \eqref{eq:sdhcs} in the first-order form \eqref{eq:sdhcs_firstorder}. Let $\vec{W}=\left(\vec{U},\vec{V}\right)^\top$. The system becomes 
	\[\frac{\mathrm{d}}{\mathrm{d}t}\vec{W} = M\vec{W} + \vec{F}(t),\quad \vec{F}=
    \begin{bmatrix}
        0\\
        e_1f_1(t)+e_{n-2}f_n(t)
    \end{bmatrix}\]
	where
	\[M = 
	\begin{bmatrix}
		0 & I \\
		A & \mathrm{diag}(-\beta)
	\end{bmatrix}\]
	and both $f_1(t)$ and $f_n(t)$ are $T$-periodic. The solution is given by the variation of parameters formula (see for example \cite{PH02}),
	\[\vec{W}(t) = e^{Mt}\vec{W}(0)+\int_0^te^{M(t-s)}\vec{F}(s)\mathrm{d}s.\]
	For $\vec{W}(t)$ to be $T$-periodic, we require the periodicity condition $\vec{W}(T)=\vec{W}(0)$. Substituting $t=T$ and rearranging we find
	\[\left(I-e^{MT}\right)\vec{W}(0)=\int_0^Te^{M(T-s)}\vec{F}(s)\;\mathrm{d}s.\]
			Since $\beta>0$ and $\mu_k<0$ all eigenvalues $\lambda_k^\pm$ of $M$ in Proposition \ref{prop:esystem} have strictly negative real parts. Thus the eigenvalues of the matrix exponential $e^{MT}$ have magnitude strictly less than 1. This implies the matrix $(I-e^{MT})$ is invertible. Thus there exists a unique initial vector $\vec{W}^*(0)$ that generates the periodic orbit.
		\item Let $\vec{W}(t)$ be any arbitrary solution and $\vec{W}^*(t)$ be the periodic solution. Define the error $\vec{E}(t)=\vec{W}(t) - \vec{W}^*(t)$. The error satisfies the homogeneous equation
			\[\frac{\mathrm{d}}{\mathrm{d}t}\vec{E}=M\vec{E}(t).\]
			Hence $\vec{E}(t)=e^{Mt}\vec{E}(0)$, which decays exponentially to zero as $t\to\infty$ by Corollary \ref{cor:expdecay}. Thus the periodic solution acts as a global attractor.
	\end{enumerate}
\end{proof}

Figure \ref{fig:damped_time_periodic} illustrates both parts of Proposition \ref{prop:tperiodicdamped}. In panel (a) the initial data is chosen to lie on the periodic orbit $\vec{U}^*$, hence the solution stays on the orbit for all $t$. In panel (b) the initial curve is the straight line interpolant at rest and the trajectory converges to the same orbit as predicted by the global asymptotic stability of (ii).

\begin{figure}[htb]
	\centering
	\begin{subfigure}[t]{0.48\textwidth}
		\includegraphics[width=\textwidth]{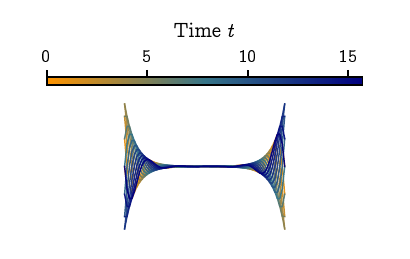}
		\caption{Evolution of time periodic orbit $\vec{U}^*$ with the initial curve on the orbit.}
	\end{subfigure}
	\hfill
	\begin{subfigure}[t]{0.48\textwidth}
		\includegraphics[width=\textwidth]{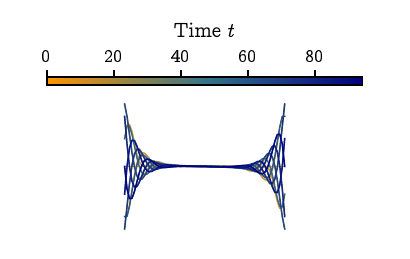}
		\caption{Evolution of a straight line joining the boundary points converging to the periodic orbit in (a).}
	\end{subfigure}
	\caption[Examples of time periodic solutions with damping.]{Examples of time periodic solutions with damping. Time steps are shown superimposed where the colour changes from yellow through teal to navy blue indicates the progression of time $t$. Parameters if the curves are given in Appendix \ref{sec:app}.}
	\label{fig:damped_time_periodic}
\end{figure}

\subsection{Solutions for general initial data}\label{sec:generalsolutions}
The previous two sections studied solutions to \eqref{eq:sdhcs} with some particular structure, being self-similar or time periodic solutions. We now turn to the general solution to \eqref{eq:sdhcs} for arbitrary initial data $\vec{X}_0=\vec{X}(0)$ and arbitrary time-dependent boundary forcing $f_1(t),f_n(t)$.The linearity of \eqref{eq:sdhcs} makes this a routine application of the variation of parameters formula \emph{once} the spectral analysis of Section \ref{sec:mat} is in hand.
\begin{thm}[Solution for general initial data]\label{thm:generalinitialdata}
	Let $\beta\geqslant0$. Given an initial curve $\vec{X}_0=\vec{X}(0)$ with $n$ points in $\mathbb{R}^p$ with corresponding $\vec{U}_0 = \vec{U}(0)$, a prescribed initial velocity $\vec{V}_0=\dot{\vec{{U}}}(0)$ and end-point boundary conditions $X_1(t)=f_1(t)$ and $X_n(t)=f_n(t)$ with $f_1,f_n$ locally integrable, \eqref{eq:sdhcs} has a unique mild solution, which is classical if $f_1,f_n$ are continuous given by
	\begin{equation}\label{eq:generalinitial}
		\begin{bmatrix}
			\vec{U}(t)\\
			\vec{V}(t)
		\end{bmatrix}
		=
		e^{Mt}
		\begin{bmatrix}
		\vec{U}_0\\
		\vec{V}_0
		\end{bmatrix}\
		+
		\int_0^te^{M(t-s)}\vec{F}(s)\;\mathrm{d}s,
	\end{equation}
    where $e^{Mt}$ is given by \eqref{eq:matexp}. Equivalently,
    \begin{equation}\label{eq:generalsolutionunpacked}
        \vec{U}(t)=\mathcal{C}(t)\vec{U}_0+\mathcal{S}(t)\vec{V}_0+\int_0^t\mathcal{S}(t-s)\vec{f}(s)\;\mathrm{d}s.
    \end{equation}
\end{thm}
\begin{proof}
	We write the system \eqref{eq:sdhcs} in the first order form \eqref{eq:sdhcs_firstorder} so
	\[\dot{\vec{W}}=M\vec{W}+\vec{F}(t)\]
	where $\vec{W}=(\vec{U},\vec{V})^\top,\vec{V}=\dot{\vec{U}}$ and
	\[M=
	\begin{bmatrix}
		0 & I \\
		A & \mathrm{diag}(-\beta)
	\end{bmatrix}.\]
	By the variation of parameters formula for linear systems, see for example \cite[Chapter IV Section 2]{PH02}, the unique solution is
	\[\vec{W}(t) = e^{Mt}\vec{W}(0)+\int_0^te^{M(t-s)}\vec{F}(s)\;\mathrm{d}s,\]
	which is \eqref{eq:generalinitial}. For locally integrable forcing this is the unique mild solution and it is classical when the forcing is continuous. Since $\vec{F}(s)=(0,\vec{f}(s))^\top$, the first block row of \eqref{eq:matexp} gives \eqref{eq:generalsolutionunpacked}.
\end{proof}

Theorem \ref{thm:generalinitialdata} furnishes the general solution in its fullest generality, accommodating arbitrary time-dependent boundary forcing. In many situations of interest, however, the boundary points are held fixed throughout the evolution. The integral term in \eqref{eq:generalsolutionunpacked} can then be evaluated in closed form, eliminating the convolution and yielding an expression in terms of $\mathcal{C},\mathcal{S}$ and the static equilibrium alone.
\begin{cor}[Solution for general initial data]\label{cor:generalinitialdataconstant}
	Let $\beta\geqslant0$. Given an initial curve $\vec{X}_0=\vec{X}(0)$ with $n$ points in $\mathbb{R}^p$, initial velocity $\vec{V}_0$ and constant boundary values $f_1(t)=C_1$ and $f_n(t)=C_n$ \eqref{eq:sdhcs} has a unique solution given by
	\begin{equation}\label{eq:generalinitialdataconstant}
		\begin{bmatrix}
			\vec{U}(t)\\
			\vec{V}(t)
		\end{bmatrix}
		=
		e^{Mt}
        \begin{bmatrix}
            \vec{U}_0\\
            \vec{V}_0
        \end{bmatrix}
        +
        M^{-1}
        \left(e^{Mt}-I\right)
        \begin{bmatrix}
            0\\
            e_1C_1 + e_{n-2}C_n
        \end{bmatrix}
	\end{equation}
    where $e^{Mt}$ is given by \eqref{eq:matexp}. Equivalently
    \begin{equation}\label{eq:generalinitialdataconstantunpacked}
        \vec{U}(t)=\mathcal{C}(t)\vec{U}_0+\mathcal{S}(t)\vec{V}_0+A^{-1}(\mathcal{C}(t)-I)(e_1C_1+e_{n-2}C_n).
    \end{equation}
\end{cor}
\begin{proof}
    When the forcing is constant $\vec{F}(s)=\vec{F}=(0,e_1C_1+e_{n-2}C_n)^\top$ and the integral term of \eqref{eq:generalinitial} becomes
    \[\int_0^te^{M(t-s)}\vec{F}\;\mathrm{d}s=\left(\int_0^te^{Mr}\;\mathrm{d}r\right)\vec{F}.\]
   By Proposition \ref{prop:esystem} every eigenvalue  $\lambda$ of $M$ satisfies $\lambda^2+\beta\lambda=\mu_k<0$ for some $k$. Thus $\lambda\neq0$ and $M$ is invertible. Thus
   \[\int_0^te^{Mr}\;\mathrm{d}r=M^{-1}\left(e^{Mt}-I\right)\]
   giving us \eqref{eq:generalinitialdataconstant}. For \eqref{eq:generalinitialdataconstantunpacked} the integral term of \eqref{eq:generalsolutionunpacked} becomes
   \[\left(\int_0^t\mathcal{S}(r)\;\mathrm{d}r\right)\left(e_1C_1+e_{n-2}C_n\right).\]
   From the proof of Proposition \ref{prop:matexp}, $\mathcal{C}'=A\mathcal{S}$ and $\mathcal{C}(0)=I,$ so
   \[A\int_0^t\mathcal{S}(r)\;\mathrm{d}r=\mathcal{C}(t)-I.\]
\end{proof}

Figure \ref{fig:general_initial_data} shows the solution from the same initial curve with a fixed boundary and with a decaying oscillatory boundary. 

\begin{figure}[htb]
	\centering
	\begin{subfigure}[t]{0.48\textwidth}
		\includegraphics[width=\textwidth]{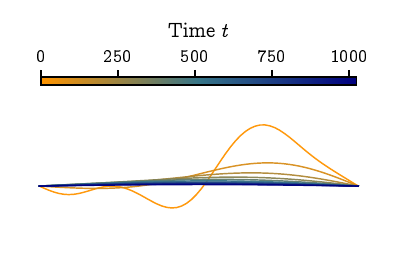}
		\caption{Constant boundary.}
	\end{subfigure}
	\hfill
	\begin{subfigure}[t]{0.48\textwidth}
		\includegraphics[width=\textwidth]{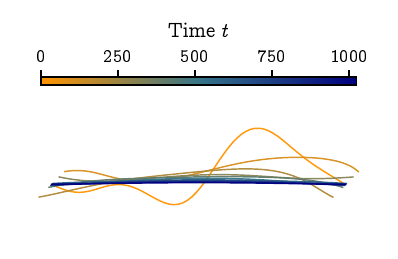}
		\caption{Time-dependent boundary.}
	\end{subfigure}
	\caption[Example solutions for general initial data.]{Example solutions to general initial data with constant boundary and a time-dependent boundary. Time steps are shown superimposed where the colour change from yellow through teal to navy blue indicates the progression of time $t$. Parameters of the curves are given in Appendix \ref{sec:app}.}
	\label{fig:general_initial_data}
\end{figure}

Corollary \ref{cor:generalinitialdataconstant} exhibits an explicit solution formula in which every quantity is determined directly from the initial data and the matrix exponential $e^{Mt}$, with no integration required. The structure of \eqref{eq:generalinitialdataconstant} clarifies the long-time behaviour of solutions with constant boundary forcing.
\begin{rem}
    If $\beta>0$, then $e^{Mt}\to0$ exponentially as $t\to\infty$ by Corollary \ref{cor:expdecay}. Hence also $\mathcal{C}(t),\mathcal{S}(t)\to0.$ In \eqref{eq:generalinitialdataconstant} the homogeneous terms therefore decay exponentially, while the particular term converges to the finite limit \[\vec{U}_\infty=-A^{-1}(e_1C_1+e_{n-2}C_n),\] 
	which is precisely the static equilibrium $\vec{U}_\infty=-A^{-1}\alpha$, the straight line interpolant between the boundary points. Note however, if $\beta=0$ the energy is conserved by Proposition \ref{prop:discreteenergy} so no such convergence occurs.
\end{rem}
\section{Semi-discrete flow between curves with boundary}\label{sec:curvaturedifference} 
The problem of flowing one geometric object to another by a curvature driven evolution was first raised by Yau and subsequently analysed by Lin and Tsai \cite{YL09}. The motivating question is, given two curves $\gamma_0$ and $\gamma_1$ in the plane, under what conditions can one construct a curvature flow that evolves $\gamma_0$ to $\gamma_1$? Lin and Tsai gave the systematic treatment for convex closed plane curves, under the additional assumptions that $\gamma_0$ and $\gamma_1$ have the same length and that curvature of the evolving curve remained uniformly bounded, the length preserving linear flow
\[\frac{\partial X}{\partial t}=\left(\frac{1}{\kappa}-\frac{1}{\overline{\kappa}}\right)N\]
evolves $\gamma_0$ exponentially to $\gamma_1$, where $\kappa$ and $\overline{\kappa}$ are the curvatures of the evolving and target curves respectively.

The semi-discrete parabolic case with boundary was treated by the authors in \cite{JM24_1}, where the parabolic curvature difference flow
\[\dot{\vec{U}} = A(\vec{U}-\vec{W})+\widetilde{\vec{f}}(t)\]
drives a piecewise linear curve $\vec{X}_0$ to any other piecewise linear curve $\vec{Y}$ with the same number of vertices.

The present section addresses the remaining case, the semi-discrete hyperbolic flow with boundary. We construct a curvature difference flow which, for $\beta>0$ and boundary trajectories converging to the endpoints of the target, converges as $t\to\infty$ from any piecewise linear curve $\vec{X}_0$ to any other piecewise linear curve $\vec{Y}$ with the same number of vertices.
\begin{defn}[Curvature difference flow]\label{defn:curvature_difference}
	Let $\vec{X}_0$ and $\vec{Y}$ be piecewise linear curves with $n$ vertices, having interior points $\vec{U},\vec{W}$ and boundary points $X_1(0),X_n(0)$ and $Y_1,Y_n$ respectively. Let $f_1,f_n$ be prescribed boundary trajectories satisfying $f_1(0)=X_1(0)$ and $f_n(0)=X_n(0)$. The semi-discrete hyperbolic curvature difference flow from $\vec{X}_0$ to $\vec{Y}$ is the system
	\begin{equation}\label{eq:curvature_difference}
		\ddot{\vec{U}}+\beta\dot{\vec{U}}=A(\vec{U}-\vec{W})+\widetilde{\vec{f}}(t)
	\end{equation}
	where 
	\[\widetilde{\vec{f}}(t) = e_1(f_1(t)-Y_1) + e_{n-2}(f_n(t)-Y_n).\]
\end{defn}

The convergence of the boundary forcing term follows from the lemma below.
\begin{lem}[Stable convolution]\label{lem:convolution}
    Let $\Phi:[0,\infty)\to\mathbb{R}^{d\times d}$ be continuous with $||\Phi(t)||\leqslant c_0e^{-\delta t}$ for some $c_0,\delta>0$ and let $\vec{g}:[0,\infty)\to\mathbb{R}^{d\times p}$ be continuous with $\vec{g}(s)\to0$ as $s\to\infty$. Then
    \[\int_0^t\Phi(t-s)\vec{g}(s)\;\mathrm{d}s\to0\text{ as }t\to\infty.\]
\end{lem}
\begin{proof}
    Since $||\Phi||$ is integrable on $[0,\infty)$ and $\vec{g}$ is bounded, continuous and tends to zero, this is a special case of \cite[Chapter 2, Theorem 2.2(i)]{GLS90}.
\end{proof}

\begin{thm}[Curvature difference flow solution]\label{thm:yau}
	Let $\vec{X}_0$ and $\vec{Y}$ be piecewise linear curves with $n$ vertices and corresponding interiors $\vec{U}_0$ and $\vec{W}$ as in Definition \ref{defn:curvature_difference}. Let $\beta>0$, $\vec{V}_0\in\mathbb{R}^{(n-2)\times p}$ and $f_1,f_n$ be continuous prescribed boundary trajectories satisfying $f_1(0)=X_1(0),f_n(0)=X_n(0)$ and
	\begin{equation}\label{eq:yaubd}
		f_1(t)\to Y_1\text{ and }f_n(t)\to Y_n\text{ as }t\to\infty.
	\end{equation}
	Then the curvature difference flow \eqref{eq:curvature_difference} with initial conditions $\vec{U}(0)=\vec{U}_0$ and $\dot{\vec{U}}(0)=\vec{V}_0$ admits a unique solution defined for all $t\geqslant0$, given explicitly by
	\begin{equation}\label{eq:yauthm}
		\begin{bmatrix}
			\vec{U}(t)\\
			\vec{V}(t)
		\end{bmatrix}
		=
        \begin{bmatrix}
            \vec{W}\\
            0
        \end{bmatrix}
        +
        e^{Mt}
        \begin{bmatrix}
            \vec{U}_0-\vec{W}\\
            \vec{V}_0
        \end{bmatrix}
        +
        \int_0^te^{M(t-s)}
        \widetilde{\vec{F}}(s)\;\mathrm{d}s,\quad
        \widetilde{\vec{F}}(s)=
        \begin{bmatrix}
            0\\
            \widetilde{\vec{f}}(s)
        \end{bmatrix}
	\end{equation}
	where $e^{Mt}$ is given by \eqref{eq:matexp}. Equivalently
    \begin{equation}\label{eq:yauthmunpacked}
        \vec{U}(t)=\vec{W}+\mathcal{C}(t)\left(\vec{U}_0-\vec{W}\right)+\mathcal{S}(t)\vec{V}_0+\int_0^t\mathcal{S}(t-s)\widetilde{\vec{f}}(s)\;\mathrm{d}s.
    \end{equation}
    Moreover $\vec{U}(t)\to\vec{W}$ and $\dot{\vec{U}}(t)\to0$ as $t\to\infty$ and consequently $\vec{X}(t)\to\vec{Y}$.
\end{thm}
\begin{proof}
	Define $\vec{Z}=\vec{U}-\vec{W}$. Then $\dot{\vec{Z}}=\dot{\vec{U}}$ and $\ddot{\vec{Z}}=\ddot{\vec{U}}$. Substituting into \eqref{eq:curvature_difference} we obtain
	\[\ddot{\vec{Z}}+\beta\dot{\vec{Z}}=A\vec{Z} + e_1(f_1(t)-Y_1) + e_{n-2}(f_n(t)-Y_n).\]
	This is precisely \eqref{eq:sdhcs} for $\vec{Z}$ with modified boundary trajectories $\widetilde{f}_1(t)=f_1(t)-Y_1$ and $\widetilde{f}_n(t)=f_n(t)-Y_n$. Applying Theorem \ref{thm:generalinitialdata} to $\vec{Z}$ with $\vec{Z}(0)=\vec{U}_0-\vec{W}$ and $\dot{\vec{Z}}(0)=\vec{V}_0$ and translating back via $\vec{U}=\vec{Z}+\vec{W}$ yields \eqref{eq:yauthm} and \eqref{eq:yauthmunpacked}. For convergence, since $\beta>0$, Corollary \ref{cor:expdecay} gives $||e^{Mt}||\leqslant c_\delta e^{-\delta t}$ for some $c_\delta,\delta>0$, thus:
	\begin{enumerate}
		\item The homogeneous term $e^{Mt}\left(\vec{U}_0-\vec{W},\vec{V}_0\right)^\top$ decays to zero.
		\item For the non-homogeneous part $\widetilde{F}
        (t)\to0$ as $t\to\infty$ by \eqref{eq:yaubd}, so the integral term converges to zero by Lemma \ref{lem:convolution}.
	\end{enumerate}
	Therefore $\vec{U}(t)\to\vec{W}$ and $\dot{\vec{U}}(t)\to0$ as $t\to\infty$ giving $\vec{X}(t)\to\vec{Y}$.
\end{proof}
\begin{rem}
	\leavevmode
	\begin{enumerate}
		\item The non-uniqueness is \emph{parametric} in the sense that the ``shape'' of the interpolation between $\vec{X}_0$ and $\vec{Y}$ is governed by the choice of boundary trajectories. This is analogous to choosing different homotopies between two paths.
		\item We have restricted here to the case where the initial and target curves have the same number of vertices. If instead this were not the case, one can duplicate vertices or add vertices along the line segments.
		\item One can also consider a target curve that is moving rather than stationary.
	\end{enumerate}
\end{rem}
The cleanest special case of Theorem \ref{thm:yau} arises when the initial and target curves share the same boundary points and the boundary is held fixed through the evolution. The boundary forcing term vanishes, the integral term drops out and the flow reduces to the homogeneous decay of the displacement $\vec{Z}=\vec{U}-\vec{W}$.
\begin{cor}[Curvature difference flow solution]
	Let $\vec{X}_0$ and $\vec{Y}$ be piecewise linear curves with $n$ vertices sharing the same boundary points, that is $X_1(0)=Y_1$ and $X_n(0)=Y_n$. If the boundary points are held fixed, $f_1(t)=Y_1$ and $f_n(t)=Y_n$ for all $t\geqslant0$, then $\widetilde{\vec{f}}\equiv0$ and the curvature difference flow \eqref{eq:curvature_difference} with $\vec{U}(0)=\vec{U}_0$ and $\dot{\vec{U}}(0)=\vec{V}_0$ has the unique solution 
	\begin{equation}\label{eq:yaucor}
	\begin{bmatrix}
		\vec{U}(t)\\
		\vec{V}(t)
	\end{bmatrix}
	=
	\begin{bmatrix}
		\vec{W}\\
		0\\
	\end{bmatrix}
	+
	e^{Mt}
	\begin{bmatrix}
		\vec{U}_0-\vec{W}\\
		\vec{V}_0
	\end{bmatrix},
	\end{equation}
    equivalently $\vec{U}(t)=\vec{W}+\mathcal{C}(t)\left(\vec{U}_0-\vec{W}\right)+\mathcal{S}(t)\vec{V}_0$, which exists for all time and if $\beta>0$ converges exponentially to $\vec{W}$, hence $\vec{X}$ converges to $\vec{Y}$.
\end{cor}
\begin{proof}
    Since $X_1(0)=Y_1$ and $X_n(0)=Y_n$ the forcing vanishes because
	\[\widetilde{f}_1(t)=f_1(t)-Y_1=0\text{ and }\widetilde{f}_n(t)=f_n(t)-Y_n=0,\]
	so the integral term in \eqref{eq:yauthm} and \eqref{eq:yauthmunpacked} vanishes identically and the solution reduces to \eqref{eq:yaucor}. 
\end{proof}
Figure \ref{fig:between_curves} demonstrates the curvature difference flow in two separate cases, one where the boundary is fixed and the other where the boundary decays to the target boundary.

\begin{figure}[htb]
	\centering
	\begin{subfigure}[t]{0.48\textwidth}
		\includegraphics[width=\textwidth]{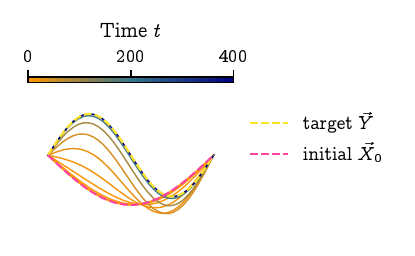}
		\caption{Constant boundary.}
	\end{subfigure}
	\hfill
	\begin{subfigure}[t]{0.48\textwidth}
		\includegraphics[width=\textwidth]{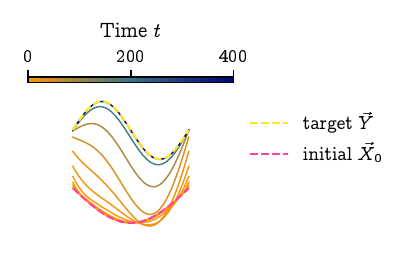}
		\caption{Time dependent boundary.}
	\end{subfigure}
	\caption[Example solutions to the curvature difference flow.]{Example solutions of the Yau type curvature difference flow with constant boundary or time dependent boundary. The target and initial curves are shown dashed. Time steps are superimposed where the colour change from yellow through teal to navy blue indicates the progression of time. Parameters of the curves are given in Appendix \ref{sec:app}.}
	\label{fig:between_curves}
\end{figure}
\section{Variants of the system}\label{sec:variants}

\subsection{Coordinate-wise anisotropy}
The coordinate-wise variant is the most direct anisotropic extension of \eqref{eq:sdhcs} since each ambient coordinate is given its own \emph{spring stiffness} $c_j$ while the spatial coupling structure encoded by $A$ is unchanged. Its tractability stems from the observation that the matrix $A$ acts on the vertex index only, not on the ambient coordinates.
\begin{defn}[Coordinate-wise anisotropic hyperbolic flow]
	Let $c_1,c_2,\dots,c_p>0$ be anisotropy constants. The coordinate-wise anisotropic semi-discrete hyperbolic flow is
	\begin{equation}\label{eq:coordwiseani}
		\ddot{U}^j + \beta\dot{U}^j = c_j\left(AU^j+e_1f^j_1(t) + e_{n-2}f^j_n(t)\right)\text{ for }j=1,2,\dots,p,
	\end{equation}
    where $\vec{U}^j\in\mathbb{R}^{n-2}$ is the $j$-th column of $\vec{U}$. Equivalently in matrix form
	\begin{equation*}
		\ddot{\vec{U}}+\beta\dot{\vec{U}}=\left(A\vec{U} + \vec{f}(t)\right)\mathrm{diag}(c_1,c_2,\dots,c_p).
	\end{equation*}
    The boundary contributions are scaled by $c_j$ as well, so that each coordinate of every interior vertex is accelerated by $c_j$ times the corresponding coordinate of its discrete curvature $N_i$.
\end{defn}
\begin{rem}
	The system \eqref{eq:coordwiseani} decouples by coordinate $j$. Each coordinate evolves independently under \eqref{eq:sdhcs} with coefficient matrix $c_jA$ and forcing $c_j\left(e_1f^j_1+e_{n-2}f_n^j\right)$.
\end{rem}

\begin{prop}[Coordinate-wise anisotropy]\label{prop:coordinatewise}
    Let $c_j>0$ and $\beta\geqslant0$, then let
    \[M_j=
    \begin{bmatrix}
        0 & I \\
        c_jA & -\beta I
    \end{bmatrix},\]
    and let $\mathcal{S}_j,\mathcal{C}_j$ be defined by \eqref{eq:matsol} with $\mu_k$ replaced by $c_j\mu_k$. Then, given an initial curve $\vec{X}_0=\vec{X}(0)$ with $n$ vertices in $\mathbb{R}^p$, initial velocity $\vec{V}_0$ and $X_1(t)=f_1(t)$ and $X_n(t)=f_n(t)$, the $j$-th coordinate of the unique solution to \eqref{eq:coordwiseani} is
	\begin{align*}
	\begin{bmatrix}
		\vec{U}^j(t) \\
		\vec{V}^j(t)
	\end{bmatrix}
	&=
	e^{M_jt}
	\begin{bmatrix}
		\vec{U}^j_0\\
		\vec{V}^j_0\\
	\end{bmatrix}
	+
    \int_0^t e^{M_j(t-s)}\vec{F}_j(s)\;\mathrm{d}s,
	\end{align*}
    where
    \[e^{M_jt}=
    \begin{bmatrix}
        \mathcal{C}_j(t) & \mathcal{S}_j(t) \\
        c_jA\mathcal{S}_j(t) & \mathcal{S}'_j(t)
    \end{bmatrix}
    \text{ and }
    \vec{F}_j(s)=
    \begin{bmatrix}
        0\\
        c_j(e_1f^j_1(s)+e_{n-2}f_n^j(s))
    \end{bmatrix}.\]
    Equivalently,
    \begin{equation}\label{eq:coordinatewiseunpacked}
        \vec{U}^j(t)=\mathcal{C}_j(t)\vec{U}_0^j+\mathcal{S}_j(t)\vec{V}_0^j+c_j\int_0^t\mathcal{S}_j(t-s)\left(e_1f_1^j(s)+e_{n-2}f_n^j(s)\right)\;\mathrm{d}s.
    \end{equation}
    The modal exponents of the $j$-th coordinate are
    \[\lambda_{k,j}^\pm=\frac{-\beta\pm\sqrt{\beta^2+4c_j\mu_k}}{2},\quad k=1,2,\dots,n-2.\]
\end{prop}
\begin{proof}
	The proof follows in the same manner as Theorem \ref{thm:generalinitialdata}.
\end{proof}

\begin{figure}[htb]
	\centering
	\begin{subfigure}[t]{0.48\textwidth}
		\includegraphics[width=\textwidth]{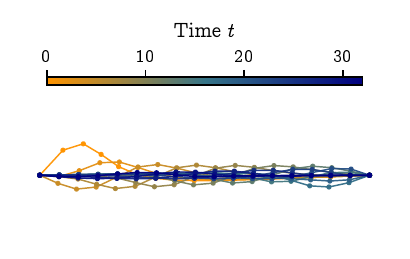}
		\caption{Isotropic.}
	\end{subfigure}
	\hfill
	\begin{subfigure}[t]{0.48\textwidth}
		\includegraphics[width=\textwidth]{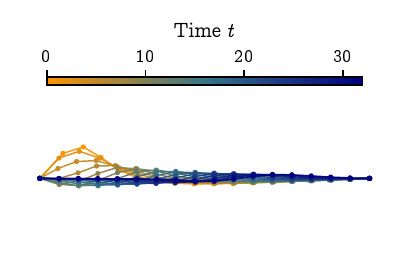}
		\caption{Anisotropic, slower in the vertical direction.}
	\end{subfigure}
	\caption{Example evolutions of the coordinate-wise anisotropic flow. Time steps are superimposed where the colour change from yellow through teal to navy blue indicates the progression of time. Parameters of the curves are given in Appendix \ref{sec:app}.}
	\label{fig:coordinate_wise}
\end{figure}

Figure \ref{fig:coordinate_wise} compares the isotropic flow $(c_1,c_2)=(1,1)$ with the anisotropic case $(c_1,c_2)=(1,1/8)$, producing the asymmetric relaxation visible in panel (b).

\subsection{Vertex-wise anisotropy}
The vertex-wise variant is the natural model for non-uniform discrete media, such as a mass-spring chain with different masses at each vertex, a serial manipulator with non-uniform link inertias, and others. The coefficient matrix $T$ is obtained by left-multiplying $A$ by a positive diagonal matrix, the result is neither Toeplitz nor symmetric, but inherits from $A$ the property of being tridiagonal.
\begin{defn}[Vertex-wise anisotropic hyperbolic flow]
	Let $c_2,c_3,\dots,c_{n-1}>0$ be \emph{mobility constants}. The vertex-wise anisotropic semi-discrete hyperbolic flow is
	\begin{equation}
		\ddot{\vec{U}}(t)+\beta\dot{\vec{U}}(t)=T\vec{U}+\mathrm{diag}(c_2,c_3,\dots,c_{n-1})\vec{f}(t)
	\end{equation}
	where $T=\mathrm{diag}(c_2,c_3,\dots,c_{n-1})A$, so that the $i$-th interior vertex is accelerated by $c_i$ times its discrete curvature $N_i$.
\end{defn}
\begin{rem}
	Unlike the coordinate-wise case, the vertex-wise system does \emph{not} decouple. The coefficient matrix $T$ is neither  Toeplitz nor symmetric, so the eigenvalue formulae from section \ref{sec:mat} do not apply directly. However, $T$ is still tridiagonal and as such this structure can be exploited.
\end{rem}

\begin{prop}[Eigenvalue properties]\label{prop:vertexwise}
	Let $T=\mathrm{diag}(c_2,c_3,\dots,c_{n-1})A$ with all $c_i>0$. Then:
	\begin{enumerate}[label=(\roman*)]
		\item The eigenvalues $\mu_1,\mu_2,\dots,\mu_{n-2}$ of $T$ are real, negative, and simple.
		\item The first-order system matrix
			\[
			M=
			\begin{bmatrix}
				0 & I\\
				T & -\beta I
			\end{bmatrix}
			\]
			has, counted with algebraic multiplicity $2(n-2)$ eigenvalues
			\[\lambda_k^\pm = \frac{-\beta\pm\sqrt{\beta^2+4\mu_k}}{2},\text{ for }k=1,2,\dots,n-2.\]
		\item If $\beta>0$, all eigenvalues of $M$ have a strictly negative real part.
	\end{enumerate}
\end{prop}
\begin{proof} 
	\leavevmode
	\begin{enumerate}[label=(\roman*)]
		\item This is established in \cite[Proposition 5.2]{JM24_1}.
		\item Let $\mu_k$ be an eigenvalue of $T$ with eigenvector $u_k$. Setting $v=(u_k,\lambda u_k)^\top$, the eigenvalue equation $Mv=\lambda v$ reduces to
			\[Tu_k=(\lambda^2+\beta\lambda)u_k\implies \lambda^2+\beta\lambda-\mu_k=0.\]
		\item Let $\beta>0$. Since $\mu_k<0$ the discriminant is $\Delta_k=\beta^2+4\mu_k=\beta^2-4|\mu_k|$. If $\Delta_k\geqslant0$ both roots are real and using Viete's formulae (see for example \cite{VI46}) we know $\lambda_k^++\lambda_k^-=-\beta<0$ and $\lambda_k^+\lambda_k^-=-\mu_k>0$ thus both roots are negative. If $\Delta_k<0$ the roots are complex with real part $-\beta/2<0$.
	\end{enumerate}
\end{proof}

Figure \ref{fig:vertex_wise} compares the isotropic flow $(c_i\equiv 1)$ with the anisotropic case where the central vertices have a high mobility (determined by a Gaussian centered on the curve).
\begin{figure}[htb]
	\centering
	\begin{subfigure}[t]{0.48\textwidth}
		\includegraphics[width=\textwidth]{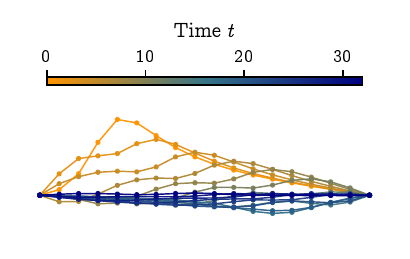}
		\caption{Isotropic.}
	\end{subfigure}
	\hfill
	\begin{subfigure}[t]{0.48\textwidth}
		\includegraphics[width=\textwidth]{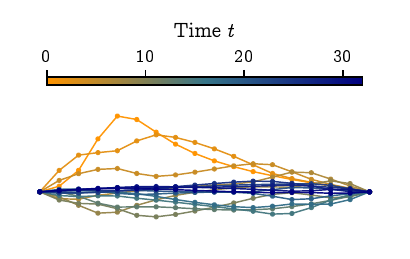}
		\caption{Anisotropic, faster central vertices.}
	\end{subfigure}
	\caption{Example evolutions of the vertex-wise anisotropic flow. Time steps are superimposed where the colour change from yellow through teal to navy blue indicates the progression of time. Parameters of the curves are given in Appendix \ref{sec:app}.}
	\label{fig:vertex_wise}
\end{figure}

\subsection{Coefficient matrix variants}
The two anisotropic variants above modify \eqref{eq:sdhcs} by rescaling either the ambient coordinates or the vertices themselves, whilst leaving the structure of the coefficient matrices intact. A third natural direction is to vary the coefficient matrices themselves, replacing the scalar damping $\beta I$ and the spatial operator $A$ by general matrices $D$ and $T$ respectively. Theorem \ref{thm:generalinitialdata} extends, namely to any pair $(T,D)$ that are simultaneously diagonalisable, the exact setting anticipated by Proposition \ref{prop:blockdiaggeneral} and Corollary \ref{cor:blockdiagspecific}. We develop this correspondence with two worked examples from dynamical systems that are modifications of the problems presented in \cite[Section 3.9]{FT01}.

\begin{defn}\label{defn:matrixvariantsdhcs}
    Let $T$ and $D$ be $(n-2)\times(n-2)$ real matrices sharing a common eigenbasis $\{u_k\}_{k=1}^{n-2}$ with $Tu_k=\mu_ku_k$ and $Du_k=\omega_ku_k$. The \emph{matrix coefficient} semi-discrete hyperbolic flow is
    \begin{equation}\label{eq:matrixvariantflow}
        \ddot{\vec{U}}+D\dot{\vec{U}}=T\vec{U}+\vec{f}(t),
    \end{equation}
    where the forcing $\vec{f}(t)$ is an arbitrary locally integrable $(n-2)\times p$ matrix of driving terms.
\end{defn}
The flow \eqref{eq:sdhcs} is the case $T=A, D=\beta I$ with boundary forcing $\vec{f}(t)=e_1f_1(t)+e_{n-2}f_n(t)$.
\begin{prop}
    Let $T,D$ be as in Definition \ref{defn:matrixvariantsdhcs}, let
    \[\lambda_k^\pm=\frac{-\omega_k\pm\sqrt{\omega_k^2+4\mu_k}}{2},\quad k=1,2,\dots,n-2,\]
    be the eigenvalues of $B=\begin{bmatrix} 0 & I \\ T & -D\end{bmatrix}$ given by Proposition \ref{prop:blockdiaggeneral}, and let $\mathcal{S}$ and $\mathcal{C}$ be as in Remark \ref{rem:extend}. Then given interior initial data $\vec{U}_0,\vec{V}_0$ the flow \eqref{eq:matrixvariantflow} has the unique solution
    \begin{equation}
        \begin{bmatrix}
			\vec{U}(t)\\
			\vec{V}(t)
		\end{bmatrix}
		=
        e^{Bt}
		\begin{bmatrix}
		\vec{U}_0\\
		\vec{V}_0
		\end{bmatrix}
		+
		\int_0^te^{B(t-s)}\vec{F}(s)\;\mathrm{d}s,\quad\vec{F}(s)=
        \begin{bmatrix}
            0\\
            \vec{f}(s)
        \end{bmatrix},
	\end{equation}
    where $e^{Bt}$ is given by Remark \ref{rem:extend}. Equivalently
    \begin{equation}
        \vec{U}(t)=\mathcal{C}(t)\vec{U}_0+\mathcal{S}(t)\vec{V}_0+\int_0^t\mathcal{S}(t-s)\vec{f}(s)\;\mathrm{d}s.
    \end{equation}
\end{prop}
\begin{proof}
    The reduction of order and variation of parameters argument of Theorem \ref{thm:generalinitialdata} applies verbatim, with $e^{Bt}$ given by Remark \ref{rem:extend} rather than \eqref{eq:matexp}.
\end{proof}

\paragraph{A chain with ground damping}
Consider a damped mass-spring chain of $n$ masses as in Figure \ref{fig:app_ndof}, where each mass $m_i$ has scalar displacement $x_i\in\mathbb{R}$ $(p=1)$ from its equilibrium position, is connected to ground by a spring $\kappa_i$ and a viscous damper $\tau_i$, and is coupled to its neighbour by a spring of stiffness $k_i$ between mass $i$ and $i+1$. The two boundary masses follow prescribed trajectories $x_1=f_1(t)$ and $x_n=f_n(t)$, so that only the $n-2$ interior masses are free and the system has $n-2$ degrees of freedom.

\begin{figure}[htb]
	\centering
        \includegraphics[width=0.6\textwidth]{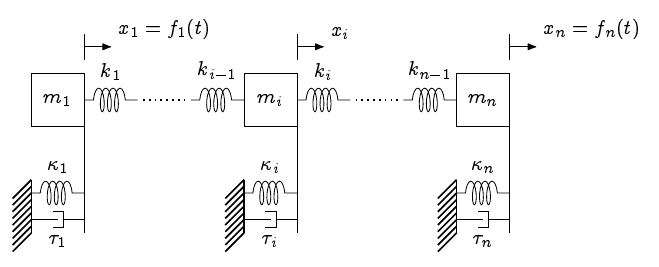}
	\caption{A chain of $n$ masses coupled to each other by springs $k_i$, coupled to the ground by springs and dampers $\kappa_i$ and $\tau_i$ and prescribed end displacements $x_1=f_1(t)$ and $x_n=f_n(t)$.}
	\label{fig:app_ndof}
\end{figure}

The force balance from Newton's second law on the $i$-th interior mass gives
\begin{multline*}
    m_i\ddot{x}_i = k_{i-1}(x_{i-1}-x_i)+k_i(x_{i+1}-x_i)-\kappa_ix_i-\tau_i\dot{x}_i \\ \iff m_i\ddot{x}_i+\tau_i\dot{x}_i=k_{i-1}x_{i-1}-(k_{i-1}+k_i+\kappa_i)x_i+k_ix_{i+1},
\end{multline*}
a near equivalent of \eqref{eq:sdhcs} in which the coefficient $-(k_{i-1}+k_i+\kappa_i)$ replaces the $-2$ of the discrete Laplacian. In matrix form, with $\vec{u}=(x_2,x_3,\dots,x_{n-1})^\top$, the diagonal mass and damping matrices $M=\mathrm{diag}(m_2,m_3,\dots,m_{n-1})$ and $C=\mathrm{diag}(\tau_2,\tau_3,\dots,\tau_{n-1})$ and the symmetric tridiagonal stiffness matrix $K$ with entries
\[K_{j,j}=-(k_j+k_{j+1}+\kappa_{j+1})\text{ and }K_{j,j+1}=K_{j+1,j}=k_{j+1}\]
the system becomes
\[M\ddot{\vec{u}}+C\dot{\vec{u}}=K\vec{u}+\vec{f}(t)\text{ where }\vec{f}(t)=k_1f_1(t)e_1+k_{n-1}f_n(t)e_n.\]
Setting $m_i=\overline{m},\tau_i=\tau,k_i=k,\kappa_i=\kappa$ and dividing by $\overline{m}$ puts this into the form of \eqref{eq:matrixvariantflow} with $T=(kA-\kappa I)/\overline{m}$ and $D=(\tau/\overline{m})I$. Note that \eqref{eq:sdhcs} can be recovered by $k=\overline{m}=1,\tau=\beta$ and $\kappa=0$.
\paragraph{A chain with inter-mass damping}
Now we add viscous dampers $d_i$ in parallel with each inter-mass coupling spring as shown in Figure \ref{fig:app_ndof_comp}.
\begin{figure}[htb]
	\centering
        \includegraphics[width=0.6\textwidth]{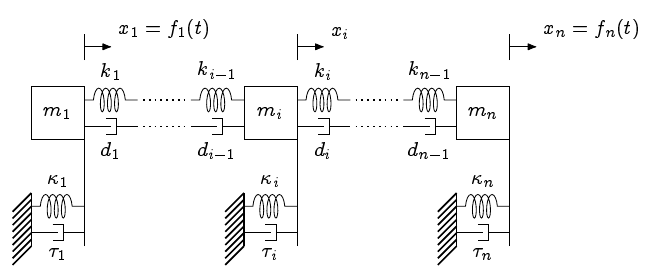}
	\caption{The chain of $n$ masses as in Figure \ref{fig:app_ndof} with the addition of viscous dampers $d_i$ in parallel with the mass coupling springs.}
	\label{fig:app_ndof_comp}
\end{figure}

The force balance on the $i$-th interior mass now picks up the additional inter-mass damping forces $d_{i-1}(\dot{x}_{i-1}-\dot{x}_i)$ and $d_i(\dot{x}_{i+1}-\dot{x}_i)$, so that
\[m_i\ddot{x}_i-d_{i-1}\dot{x}_{i-1}+(\tau_i+d_{i-1}+d_i)\dot{x}_i-d_i\dot{x}_{i+1}=k_{i-1}x_{i-1}-(k_{i-1}+k_i+\kappa_i)x_i+k_ix_{i+1},\]
so the velocity coefficients acquire the same symmetric tridiagonal damping matrix $C$ with entries 
\[C_{j,j}=\tau_{j+1}+d_j+d_{j+1}\text{ and }C_{j,j+1}=C_{j+1,j}=-d_{j+1},\]
the system becomes
\[M\ddot{\vec{u}}+C\dot{\vec{u}}=K\vec{u}+\vec{f}(t)+\vec{g}(t)\text{ where }\vec{g}(t)=d_1\dot{f}_1(t)e_1+d_{n-1}\dot{f}_n(t)e_n.\]
Setting $m_i=\overline{m},\tau_i=\tau,k_i=k,\kappa_i=\kappa,d_i=d$ and dividing by $\overline{m}$ gives \eqref{eq:matrixvariantflow} where $T=(kA-\kappa I)/\overline{m}$ and $D=(\tau I-d A)/\overline{m}$. Note the damping matrix is no longer diagonal, however both matrices are polynomials in $A$ and hence share its eigenbasis.

\section{Conclusion} \label{sec:conclusion}
To conclude we have studied the damped semi-discrete hyperbolic curvature flow \eqref{eq:sdhcs} of piecewise linear curves with prescribed endpoints. Its dynamics decompose into the $n-2$ eigenmodes of $A$, each being a damped oscillator. For $\beta>0$ every solution with fixed endpoints converges exponentially to the straight line interpolant, periodic boundary motion produces a unique globally attracting periodic response and the curvature difference flow converges from any curve to any other with the same number of vertices. For $\beta=0$ solutions with fixed endpoints are superpositions of standing waves. For plane curves, Theorem \ref{thm:selfsimilar} classifies all self-similar solutions with self-similar boundary data and a non-degenerate initial curve.

We also note the limitation of this model. The flow \eqref{eq:sdhcs} is linear, being the spatial discretisation of the damped wave equation, which is the linearisation, for small gradients and velocities of the dissipative hyperbolic mean curvature flow. The classification of self-similar solutions is limited to plane curves. Finally, the results of Sections \ref{sec:mat}-\ref{sec:curvaturedifference} are formal results whereas Section \ref{sec:variants} are more so illustrative applications of the formal results.

\subsection*{Acknowledgments}
Part of this research was completed while the first author was visiting the University of Science and Technology, China under a Chinese Academy of Sciences Presidents' International Fellowship Initiative visiting fellowship, grant number 2024PVA0042.  The authors would like to thank Stephan Tornier and Jahne Meyer for their interest in this work. We also thank the anonymous referees for their careful reading and valuable comments which led to improvements.

\begin{appendices}
    \section{Parameters for the figures} \label{sec:app}
    All the examples shown in the figures use \eqref{eq:sdhcs} with unit mesh spacing on curves in $\mathbb{R}^2$ and the abscissae $x_i=(i-1)/(n-1)$ for $i=1,2,\dots,n.$ We write $u_k$ for $k=1,2,\dots,n-2$ to represent the orthonormal eigenvectors of $A$ from Lemma \ref{lem:etri}, ordered such that $u_1$ is the slowest mode.

    \paragraph{Figure \ref{fig:energy}.}
    $n=26$ and $X_i(0)=(x_i,0.6y_{i-1})$ for $i=2,3,\dots,25$ where $y=\sum_{k=1}^{24}w_ku_k$ with $w_k=\xi_ke^{-0.4(k-1)}$ and $\xi_k$ being independent standard normal samples (NumPy \texttt{default\_rng(7)}). Panels (a)-(b) use $\beta=0.6$ and show eleven snapshots for $t\in[0,18]$. Panels (c)-(d) use $\beta=0.15$ and show eleven snapshots for $t\in[0,36]$.

    \paragraph{Figure \ref{fig:selfsimilar}.}
    In each panel the solutions are shown for $\rho=0$ in \eqref{eq:selfsimeq2}, and the profiles are the solutions of \eqref{eq:selfsimilareq4} with $b=0$ for the given endpoints. The endpoints move with the curve and the initial velocity is that of the motion.
    \begin{enumerate}[label=(\alph*)]
        \item $n=7$, $\beta=1$, $\gamma=-1/4$, $\omega_0=0$, $X_1(0)=(1,1)$, $X_7(0)=(1,-1)$. Solution is shown for five snapshots $t\in[0,1]$.
        \item $n=12$, $\beta=1/4$, $\gamma=0$, $\omega_0=1$, $X_1(0)=(0,1)$, $X_{12}(0)=(0,-1)$. Solution is shown for six snapshots $t\in[0,2\pi]$.
        \item $n=12$, $\beta=1$, $v=(0,0.1)$, $X_1(0)=(-1,0)$, $X_{12}(0)=(1,0)$. Solution is shown for six snapshots $t\in[0,12]$.
    	\item $n=9$, $\beta=1$, $\gamma=-1/4$, $\omega_0=1$, $X_1(0)=(0,1)$, $X_9(0)=(0,-1)$. Solution is shown for six snapshots $t\in[0,2\pi]$.
        \item $n=12$, $\beta=1/4$, $\gamma=0$, $\omega_0=1$, $X_1(0)=(0,0)$, $X_{12}(0)=(0,1)$, with the translation $h(t)=c(1-e^{-\beta t})$ of Lemma \ref{lem:appendtranslation}, $c=(4,0)$. Solution is shown for six snapshots $t\in[0,4\pi]$.
        \item $n=22$, $\beta=1$, $\gamma=-1/8$, $\omega_0=1$, $X_1(0)=(1,1)$, $X_{22}(0)=(1,-1)$, with the translation of (e) for $c=(1,0)$. Solution is shown for seven snapshots $t\in[0,3\pi]$.
    \end{enumerate}

    \paragraph{Figure \ref{fig:undamped_time_periodic}.}
    $\beta=0,n=34$ and $X_i(0)=(x_i,0.2\sin(\pi x_i))$. Solution is shown for 17 snapshots $t\in[0,10\pi]$.

    \paragraph{Figure \ref{fig:damped_time_periodic}.}
    $\beta=0.5,n=66$ and the $5\pi$-periodic boundary trajectories $f_1(t)=(0,0.4\sin(0.4t))$ and $f_{66}(t)=(1,-0.4\sin(0.4t))$. 
    \begin{enumerate}[label=(\alph*)]
        \item Initial data $\vec{W}(0)$ lies on the periodic orbit given by Proposition \ref{prop:tperiodicdamped}. Solution is shown for 15 snapshots $t\in[0,5\pi].$
        \item $X_i(0) = (x_i,0)$. Solution is shown for 15 snapshots $t\in[0,30\pi]$.
    \end{enumerate}

    \paragraph{Figure \ref{fig:general_initial_data}.}
    $\beta=1,n=66$ and $X_i(0)=(x_i,0.5y_{i-1})$ for $i=2,3,\dots,65$ where $y=\sum_{k=1}^{64}w_ku_k$ with $w_k=2\xi_ke^{-0.6(k-1)}$ and $\xi_k$ being independent standard normal samples (NumPy \texttt{default\_rng(42)}).  Solution is shown for eleven snapshots $t\in[0,1024]$.
    \begin{enumerate}[label=(\alph*)]
        \item $f_1\equiv(0,0)$ and $f_{66}\equiv(1,0)$.
        \item $f_1(t)=(0,0)+s(t)(1,1)$ and $f_{66}(t)=(1,0)+s(t)(1,1)$ where $s(t)=0.125e^{-0.005t}\sin(2\pi t/256)$.
    \end{enumerate}

    \paragraph{Figure \ref{fig:between_curves}.}
    $\beta=0.6,n=34$ with target curve $Y_1=(0,0.5),Y_i=(x_i,0.5+0.25\sin(2\pi x_i))$ and $Y_{34}=(1,0.5)$. Snapshots at $t=0$ and at nine geometrically spaced times from $t=1$ to $t=400$.
    \begin{enumerate}[label=(\alph*)]
        \item $X_i(0)=(x_i,0.5-0.3\sin(\pi x_i))$ with the end points fixed at $Y_1,Y_{34}$.
        \item $X_i(0)=(x_i,-0.3\sin(\pi x_i))$ and $f_j = Y_j+(X_j(0)-Y_j)e^{-t/20}$ for $j=1,34$.
    \end{enumerate}

    \paragraph{Figure \ref{fig:coordinate_wise}.}
    $\beta=0.1,n=18$ and $X_i(0)=(x_i+0.06b_{i-1},0.4b_{i-1})$  for $i=2,3,\dots,17$ where $b=\sum_{k=1}^{16}e^{-0.5(k-1)}\sin(0.3(k-1))u_k$. The solutions are shown for 16 snapshots $t\in[0,32]$.
    \begin{enumerate}[label=(\alph*)]
        \item $(c_1,c_2)=(1,1).$
        \item $(c_1,c_2)=(1,1/8).$
    \end{enumerate}
    
    \paragraph{Figure \ref{fig:vertex_wise}.}
    $\beta=0.15,n=18$ and $X_i(0)=(x_i,0.4b_{i-1})$ for $i=2,3,\dots,17$ where $b=\sum_{k=1}^{16}e^{-0.4(k-1)}\cos(0.6(k-1))u_k$. The solutions are shown for 16 snapshots $t\in[0,32].$
    \begin{enumerate}[label=(\alph*)]
        \item $c_i=1$.
        \item $c_i=1+10\exp\left(-((x_i-0.5)/0.1)^2\right)$.
    \end{enumerate}
\end{appendices}

\bibliography{sn-bibliography}


\end{document}